\documentclass[11pt]{amsart}
\usepackage{amssymb}
\usepackage{amsmath}
\usepackage{amsfonts}
\providecommand{\U}[1]{\protect\rule{.1in}{.1in}}
\theoremstyle{plain}

\newtheorem{corollary}{Corollary}

\newtheorem{lemma}{Lemma}

\newtheorem{proposition}{Proposition}
\newtheorem{remark}{Remark}

\newtheorem{theorem}{Theorem}
\numberwithin{equation}{section}

\DeclareMathOperator{\dist}{dist} \DeclareMathOperator{\diam}{diam}

\begin{document}
\title[Doubling inequality and nodal sets ]
{Doubling inequality and nodal sets for solutions of bi-Laplace equations}
\author{ Jiuyi Zhu}
\address{
Department of Mathematics\\
Louisiana State University\\
Baton Rouge, LA 70803, USA\\
Email:  zhu@math.lsu.edu }
\thanks{Zhu is supported in part by  NSF grant DMS-1656845 and OIA-1832961}
\date{}

\begin{abstract}
 We investigate the doubling inequality and nodal sets for the solutions of bi-Laplace equations. A polynomial upper bound for the nodal sets of solutions and their gradient is obtained based on the recent development of nodal sets for Laplace eigenfunctions by Logunov. In addition, we derive an implicit upper bound for the nodal sets of solutions. We show two types of doubling inequalities for the solutions of bi-Laplace equations. As a consequence, the rate of vanishing is given for the solutions.
\end{abstract}

\maketitle
\section{Introduction}

In this paper, we consider the doubling inequality and nodal sets for the solutions of bi-Laplace equations
\begin{equation}
\triangle^2 u=W(x)u \quad \mbox{in} \ \mathcal{M},
\label{bi-Laplace}
\end{equation}
where $\mathcal{M}$ is a compact and smooth Riemannian manifold with dimensions $n\geq 2$.
Assume that $\|W\|_{L^\infty}\leq M$ for some large constant $M$. The nodal sets are the zero level sets of solutions. For the eigenfunctions of Laplace
\begin{equation}
\triangle \phi_\lambda+\lambda \phi_\lambda=0
\end{equation}
on a compact smooth Riemannian manifold $\mathcal{M}$, Yau \cite{Y} conjectured that the Hausdorff measure of nodal sets satisfies
$$ c\sqrt{\lambda}\leq H^{n-1}(x\in\mathcal{M}|\phi_\lambda=0)\leq C\sqrt{\lambda},   $$
where $c$, $C$ depend on the manifold $\mathcal{M}$. The conjecture was solved in real analytic manifolds in the seminal paper by Donnelly-Fefferman \cite{DF}.  Lin \cite{Lin} provided a simpler proof for the upper
bound for general second order elliptic equations on the analytic manifolds.
For the smooth manifolds, some progresses were made towards the upper bound of nodal sets.
 On smooth surfaces, Donnelly-Fefferman \cite{DF1} showed that  $H^{1}(\{\phi_\lambda=0\})\leq C \lambda^{\frac{3}{4}} $ by using Carleman estimates and Calder\'on and Zygmund type decomposition. A different proof based on frequency functions was given by Dong \cite{D}. Recently, Logunov and
Malinnikova \cite{LM} were able to refine the upper bound to be $C \lambda^{\frac{3}{4}-\epsilon}$.  For higher dimensions $n\geq 3$, the exponential upper bound
$H^{n-1}(\{\phi_\lambda=0\})\leq C \lambda^{C\sqrt{\lambda}}$ was obtained by Hardt and Simon \cite{HS}. Very recently,  an important improvement was given by Logunov in \cite{Lo} who obtained a polynomial upper bound
$$H^{n-1}(x\in\mathcal{M}| \phi_\lambda=0)\leq C \lambda^{\alpha},$$
where $\alpha>\frac{1}{2}$ depends only on the dimension. In \cite{Lo},  Logunov further studied the frequency function of harmonic functions and developed a new combinatorial argument to investigate the nodal sets.

For the lower bound, Logunov \cite{Lo1} answered Yau's conjecture and obtained the sharp lower bound for smooth manifolds. This breakthrough improved a polynomial lower bound obtained early by Colding and  Minicozzi \cite{CM}, Sogge and Zelditch \cite{SZ}. See also the same polynomial lower bound by different methods, e.g. \cite{HSo}, \cite{M}, \cite{S}.

The upper bound of nodal sets was studied for general second order elliptic equations  in \cite{Lin}, \cite{HS}, \cite{HL1}, \cite{GR}, etc.   The Hausdorff dimension of nodal sets and singular sets for the solutions of higher order elliptic equations was studied by Han \cite{Han}. It was shown in \cite{Han} that the Hausdorff dimension of nodal sets $\{u=0\}$ and the mixed nodal sets $\{u=\triangle u=0\}$ is not greater than $n-1$, and the Hausdorff dimension of the singular sets $\{D^\nu u=0 \ \mbox{for all} \ |\nu|<4\}$ is not greater than $n-2$. Especially, the Hausdorff measure of singular sets was studied by Han, Hardt and Lin in \cite{HHL}. An implicit upper bound for the measure of singular sets in term of the doubling index was given. The optimal upper bound of nodal sets for higher order elliptic equations was obtained by Kukavica \cite{Ku} in real analytic domains. Complex analysis techniques were used for real analytic setting, which differ much from the tools in the paper.
 For the bi-Laplace equations on smooth manifolds, we want to know how the upper bound of the nodal sets depends on the potential functions appeared in the equations (\ref{bi-Laplace}). We are able to show the following result.
\begin{theorem}
Let $u$ be the solutions of bi-Laplace equations (\ref{bi-Laplace}) with $n\geq 3$. There exists a positive constant $C$ that depends only on the manifold $\mathcal{M}$ such that
$$ H^{n-1}(x\in\mathcal{M}|u=\triangle u=0)\leq CM^{\alpha}, $$
where $\alpha>\frac{1}{2}$ depends only on the dimension $n$.
\label{th2}
\end{theorem}

In all the aforementioned literature for the study of the upper bound of nodal sets of classical eigenfunctions, a crucial estimate is the following sharp quantitative doubling inequality established by Donnelly and Fefferman \cite{DF},
\begin{equation}
\|\phi_\lambda\|_{\mathbb B_{2r}(x)}\leq e^{ C\sqrt{\lambda}}\|\phi_\lambda\|_{\mathbb B_{r}(x)}
\label{doubleuse}
\end{equation}
for any $r>0$ and any $x\in \mathcal{M}$, where $\|\cdot\|_{\mathbb B_r(x_0)}$ denotes the $L^2$ norm on the ball $\mathbb B_r(x_0)$. Such optimal doubling inequalities provide the sharp upper bound for the frequency function and vanishing order for classical eigenfunctions. Roughly speaking, doubling inequalities retrieve global feature from local data. Those estimates are also widely used in inverse problems, control theorems, spectral theory, etc.

As the estimates (\ref{doubleuse}), in order to obtain  upper bound estimates of nodal sets by the norm of potential functions for the solutions of bi-Laplace equations (\ref{bi-Laplace}), we also need a quantitative doubling inequality, which provides the bounds for frequency function and rate of vanishing. We show the following doubling estimates for $u$ and $\triangle u$.

\begin{theorem}
Let $u$ be the solutions of bi-Laplace equations (\ref{bi-Laplace}). There exists a positive constant $C$ depending only on the manifold $\mathcal{M}$ such that
\begin{equation}
\|(u, \triangle u)\|_{\mathbb B_{2r}(x)}\leq e^{ CM^{\frac{2}{3}}}\|(u, \triangle u)\|_{\mathbb B_{r}(x)}
\label{LLL}
\end{equation}
for any $r>0$ and any $x\in \mathcal{M}$.
\label{th1}
\end{theorem}

If we only consider bounded potential functions for bi-Laplace equations (\ref{bi-Laplace}), the power $CM^{\frac{2}{3}}$ in the exponential functions  in (\ref{LLL}) seems to be sharp so far. Such power $CM^{\frac{2}{3}}$ appeared in the topic of quantitative unique continuation, see e.g. \cite{BK}, \cite{K}, etc. Especially, the counterexample for the sharpness of $CM^{\frac{2}{3}}$ was constructed for complex-valued potentials in \cite{Mes}. For the real-valued bounded potentials, it is still open if the sharp power is $CM^{\frac{1}{2}}$ for $n\geq 3$, which is related to Landis' conjecture \cite{KSW}.

In showing Theorem \ref{th2}, we use the doubling inequality for $(u, \triangle u)$ in Theorem \ref{th1}. Using some different type of Carleman estimates for bi-Laplace, we are able to obtain a refined  doubling inequality for the solution $u$.
\begin{theorem}
Let $u$ be the solutions of bi-Laplace equations (\ref{bi-Laplace}). There exists a positive constant $C$ depending only on the manifold $\mathcal{M}$ such that
\begin{equation}
\|u\|_{\mathbb B_{2r}(x)}\leq e^{ CM^{\frac{1}{3}}}\|u\|_{\mathbb B_{r}(x)}
\label{lala}
\end{equation}
for any $r>0$ and any $x\in \mathcal{M}$.
\label{th3}
\end{theorem}

Such type doubling inequality (\ref{lala}) without explicit dependence of potential functions was assumed by Han, Hardt and Lin in \cite{HHL} to obtain upper bounds of the measure of singular sets. Theorem \ref{th3} not only verifies that the doubling inequality holds for the solutions of bi-Laplace equations, but also provides the explicit estimates for such inequality.
As a consequence of Theorem \ref{th3}, we obtain an upper bound for the vanishing order of solutions in (\ref{bi-Laplace}). For smooth functions, vanishing order of solutions at some point is defined as the number of the highest order non-zero derivative such that all lower derivatives vanish at the point.
\begin{corollary}
Let $u$ be the solutions of bi-Laplace equations (\ref{bi-Laplace}). Then the vanishing order of solution $u$ is everywhere less than
$CM^{\frac{1}{3}}$, where $C$ depends only the manifold $\mathcal{M}$.
\label{cor1}
\end{corollary}

Our initial goal of the project is to study the upper bounds of the measure for the nodal sets $\{u=0\}$ of solutions for bi-Laplace equations. The desirable doubling inequality (\ref{lala}) is shown. However, the frequency function for bi-Laplace equations has to have $\triangle u$ involved to get an almost monotonicity result and a comparison lemma of doubling index, see section 2 for those results. These cause our upper bound estimates are on nodal sets $\{u=\triangle u=0\}$ in Theorem \ref{th2}. Inspired by the arguments for showing the finite bound of singular sets for Laplace equations in \cite{HHL} and \cite{HHL1}, we are able to derive the following  bounds for the solutions of nodal sets of bi-Laplace equations.
\begin{theorem}
Let $u$ be the solutions of bi-Laplace equations (\ref{bi-Laplace}). There exists a positive constant $C(M)$ depending only on $M$ and the manifold $\mathcal{M}$ such that
\begin{align*}
H^{n-1}(x\in\mathcal{M}| u=0)\leq C(M).
\end{align*}
\label{th4}
\end{theorem}

Let us comment on the organization of the article.  In Section 2,  we introduce the corresponding frequency function for bi-Laplace equations and obtain almost monotonicity results for the frequency function.
In section 3, the polynomial upper bound of nodal sets for bi-Laplace equations  was deduced inspired by the new combinatorial arguments in \cite{Lo}.
Section 4 is devoted to obtaining the doubling inequality for the solutions of bi-Laplace equations using Carleman estimates. A quantitative three-ball theorem is  shown. Section 5 is devoted for the study of a refined doubling inequality of solutions of bi-Laplace equations.
  In section 6, we present the proof of Theorem \ref{th4} for the measure of nodal sets.
   Section 7 is used to provide a detailed proof for a lemma on the propagation of smallness of the Cauchy data. The Appendix provides the proof of some ingredients in the arguments of Theorem \ref{th2}.
The letters $c$, $C$ and $C_i$ denote generic positive
constants that do not depend on $u$, and may vary from line to
line. The letter $M$ is assumed to be a sufficiently large positive constant.

\section{Frequency function of elliptic systems and its applications}
Frequency function was introduced by Almgren for harmonic functions. Garofalo and Lin \cite{GL}, \cite{GL1} developed the method of frequency function to study strong unique continuation property. Lin \cite{Lin} applied this tool to characterize the measure of nodal sets.  The frequency function describes the local growth rate
of the solution and is considered as a local measure of its ``degree" for a
polynomial like function. Interested readers are recommended to refer to the nice book in preparation by Han and Lin \cite{HL}. Logunov and Malinnikova \cite{LM}, \cite{Lo}, \cite{Lo1} further exploited the frequency function of harmonic functions with new combinatorial arguments to study nodal sets. In this section, we study the frequency function for bi-Laplace equations (\ref{bi-Laplace}), which lies the foundation for the combinatorial arguments in the later section.

Let us consider normal coordinates in a geodesic ball $\mathbb B_r(0)$, where $r$ is a sufficiently small. We treat the Laplace operator on the manifold as an elliptic operator in a domain in $\mathbb R^n$. For the Euclidean distance $d(x,y)$ and Riemannian distance $d_g(x, y)$, there exists a small number $\epsilon>0$ such that
$$1-\epsilon\leq \frac{d_g(x, y)}{d(x,y)}\leq 1+\epsilon   $$
for $x, y\in \mathbb B_{r_0}$ with $r_0$ depending on $\epsilon$ and the manifold.

To study the bi-Laplace equations (\ref{bi-Laplace}), we reduce it to be a system of second order elliptic equations. Let $v=\triangle u$. The solutions of (\ref{bi-Laplace}) satisfy
\begin{equation}
\left \{ \begin{array}{lll}
\triangle u= v, \medskip\\
\triangle v=W(x) u.
\end{array}
\right.
\label{system}
\end{equation}

Note that $\|W\|_{L^\infty}\leq M$ for $M>1$ sufficiently large. We do a scaling for the bi-Laplace equations  (\ref{bi-Laplace}). Let
$$\bar u(x)=u(\frac{x}{M^{1/4}}). \quad $$
Set $\bar v(x)= \triangle \bar u(x)$.
Then $\bar u(x)$ and $\bar v(x)$ satisfy
\begin{equation}
\left \{ \begin{array}{lll}
\triangle \bar u=\bar v, \medskip\\
\triangle \bar v=\bar W(x)\bar u,
\end{array}
\right.
\label{system1}
\end{equation}
where $\bar W(x)=\frac{W(x)}{M}$. Thus, $ \|\bar W\|_{L^\infty}\leq 1$. We will consider the elliptic systems (\ref{system1}) in the following sections for the nodal sets. For ease of the presentation, we still use the notations $u, v$ for $\bar u, \bar v$ in (\ref{system1}).
For the system of equations (\ref{system1}), we define the frequency function as follows,
\begin{equation}
I(x_0, r)=\frac{ r\big( \int_{\mathbb B_r(x_0)} |\nabla u|^2 + |\nabla v|^2 \,dx+ \int_{\mathbb B_r(x_0)} ({1}+\bar W(x)) uv  \,dx\big) }{\int_{\partial \mathbb B_r(x_0)} u^2 \ d\sigma+\int_{\partial \mathbb B_r(x_0)} v^2 \ d\sigma}.
\end{equation}
Without loss of generality, we may write $x_0=0$. We denote $I(x_0, r)=I(r)$.
We adopt the following notations.
\begin{align*}
&D_1(r)= \int_{\mathbb B_r} |\nabla u|^2 \,dx, \quad D_2(r)= \int_{\mathbb B_r} |\nabla v|^2 \,dx, \medskip\\
&D_3(r)= \int_{\mathbb B_r} ({1}+ \bar W(x)) uv  \,dx,  \medskip\\
& D(r)=D_1(r)+D_2(r)+D_3(r),\medskip\\
&H(r)=\int_{\partial \mathbb B_r} u^2 + v^2 \ d\sigma,
\end{align*}
where $d\sigma$ is $n-1$ dimensional Hausdorff measure on $\partial\mathbb B_r$. So we can write $I(r)$ as
\begin{equation}
I( r)=\frac{rD(r)}{H(r)}.
\label{defini}
\end{equation}
Next we want to show that the frequency function $I(r)$ is almost monotonic.
\begin{proposition}
For any $\epsilon>0$, there exists $r_0=r_0(\epsilon, \mathcal{M})$ such that
\begin{equation}
I(r_1)\leq C+ (1+\epsilon) I(r_2)
\end{equation}
for any $0<r_1<r_2<r_0$.
\label{monoto}
\end{proposition}

Before presenting the proof of proposition \ref{monoto}, we establish some elementary estimates. We adapt some the arguments from the book by Han and Lin \cite{HL}.
\begin{lemma}
There exist positive constants $r_0$ and $C$ such that
\begin{equation}
\int_{\mathbb B_r} |\nabla u|^2+ |\nabla v|^2 \, dx\leq C D(r)+CrH(r)
\label{esti1}
\end{equation}
and
\begin{equation}
D(r)\leq C\int_{\mathbb B_r} |\nabla u|^2+ |\nabla v|^2\, dx+Cr H(r).
\label{esti2}
\end{equation}
\end{lemma}
\begin{proof}
 For any $w\in H^1 (\mathbb B)$, the following estimate holds
\begin{equation}
\int_{\mathbb B_r} w^2 \, dx\leq \frac{2r}{n}\int_{\partial \mathbb B_r} w^2\, d\sigma+\frac{4r^2}{n^2} \int_{\mathbb B_r}|\nabla w|^2 \, dx.
\label{hanlin}
\end{equation}
See e.g. Lemma 3.2.2 in \cite{HL}. From the definition of $D(r)$ and the assumption of $\bar W$, using Cauchy-Schwartz inequality, we have
\begin{align*}
\int_{\mathbb B_r} |\nabla u|^2+ |\nabla v|^2 \, dx&\leq D(r)-\int_{\mathbb B_r} ({1}+ \bar W(x)) uv \, dx
\nonumber\\
&\leq D(r)+ C \int_{\mathbb B_r} u^2+ v^2 \, dx.
\end{align*}
Using the estimates (\ref{hanlin}) for $u$ and $v$, we obtain that
\begin{align}
\int_{\mathbb B_r} |\nabla u|^2+ |\nabla v|^2 \, dx
&\leq D(r)+C\big( \frac{2r}{n}\int_{\partial \mathbb B_r} u^2\, d\sigma+\frac{4r^2}{n^2} \int_{\mathbb B_r}|\nabla u|^2 \, dx \nonumber \\
 & +\frac{2r}{n}\int_{\partial \mathbb B_r} v^2\, d\sigma+\frac{4r^2}{n^2} \int_{\mathbb B_r}|\nabla v|^2 \, dx       \big).\nonumber
  \end{align}
  Since $r\in (0, \ r_0)$ for some small $r_0$, we can show that
 \begin{align}
\int_{\mathbb B_r} |\nabla u|^2+ |\nabla v|^2 \, dx &\leq CD(r)+ Cr \int_{\partial \mathbb B_r} v^2+u^2\, d\sigma\nonumber\\
&\leq CD(r)+CrH(r)
\end{align}
 for some positive constant $C$. Thus, the inequality (\ref{esti1}) is arrived. From the definition of $D(r)$, Cauchy-Schwartz inequality and (\ref{hanlin}), we can easily derive that
\begin{align}
D(r)&\leq \int_{\mathbb B_r} |\nabla u|^2+ |\nabla v|^2 \, dx+ \int_{\mathbb B_r} |{1}+\bar W(x)| |uv|\, dx \nonumber \\
&\leq \int_{\mathbb B_r} |\nabla u|^2+ |\nabla v|^2 \, dx+ C \int_{\mathbb B_r} u^2+ v^2 \, dx\nonumber \\
&\leq  \int_{\mathbb B_r} |\nabla u|^2+ |\nabla v|^2 \, dx+ C\big( \frac{2r}{n}\int_{\partial \mathbb B_r} u^2\, d\sigma+\frac{4r^2}{n^2} \int_{\mathbb B_r}|\nabla u|^2 \, dx \nonumber\\
& +\frac{2r}{n}\int_{\partial \mathbb B_r} v^2\, d\sigma+\frac{4r^2}{n^2} \int_{\mathbb B_r}|\nabla v|^2 \, dx       \big)\nonumber\\
&\leq  C \int_{\mathbb B_r} |\nabla u|^2+ |\nabla v|^2 \, dx+Cr\int_{\partial \mathbb B_r} u^2+v^2\, d\sigma.
\end{align}
Hence, this leads to the inequality (\ref{esti2}).

\end{proof}

We can check that $H(r)\not=0$ for any $r\in (0, r_0)$. If $H(r)=0$ for some $r\in (0, r_0)$, the definition of $H(r)$ implies that $u=v=0$ on $\partial\mathbb B_r$. From the elliptic systems (\ref{system1}) and integration by parts argument, we will derive $D(r)=0$. From (\ref{esti1}), then $(u, v)$ is some constant. Moreover, $u=v=0$ in $\mathbb B_r$ since $u=v=0$ on $\partial\mathbb B_r$. By the strong unique continuation property,  $u\equiv v\equiv 0$ in $\mathcal{M}$, which leads to a contradiction.
Thus,
$I(r)$ is absolutely continuous on $(0, \ r_0)$. If we set
$$ \Gamma=\{ r\in (0, \ r_0): I(r)> \max\{ 1, I(r_0)\}\},$$
then $\Gamma$ is an open set. There holds a decomposition $\Gamma=\bigcup^\infty_{j=1} (a_j, \ b_j)$ with $a_j, b_j\not\in \Gamma$. For
$r\in \Gamma$, we have $I(r)>1$, i.e.
\begin{equation}
\frac{H(r)}{r}< D(r).
\label{NNNKao}
\end{equation}

With those preparations, we are ready to give the proof of the proposition.
\begin{proof}
We will consider the derivative of $I(r)$. So we first consider derivative of $D_1(r)$, $D_2(r)$, $D_3(r)$ and $H(r)$ with respect to $r$ in some interval $(a_i, \ b_i)$. It is obvious that
\begin{equation}
D_1^{'}(r)= \int_{\partial \mathbb B_r} |\nabla u|^2 \, d\sigma.
\end{equation}
Since $|x|=r$ on $\partial \mathbb B_r$, we write
\begin{equation}
D_1^{'}(r)=\int_{\partial \mathbb B_r} |\nabla u|^2 \frac{x}{r} \cdot \frac{x}{r}    \, d\sigma.
\end{equation}
Note that the unit norm $\vec{n}$ on $\partial \mathbb B_r$ is $\frac{x}{r}$. Performing the integration by parts gives that
\begin{align*}
D_1^{'}(r)
&=\frac{1}{r} \int_{ \mathbb B_r} div (|\nabla u|^2 \cdot x) \, dx \\
&=\frac{n}{r} \int_{ \mathbb B_r}|\nabla u|^2\, dx+\frac{2}{r} \int_{ \mathbb B_r}\nabla u\cdot \nabla^2 u\cdot x\,dx\\
&=\frac{n-2}{r} \int_{ \mathbb B_r}|\nabla u|^2\, dx-\frac{2}{r}\int_{ \mathbb B_r}\triangle u\nabla u\cdot x\, dx+\frac{2}{r^2}\int_{\partial \mathbb B_r}(\nabla u\cdot x)^2 \, d\sigma.
\end{align*}
From the first equation of the elliptic systems (\ref{system1}),
\begin{equation}
D_1^{'}(r)=\frac{n-2}{r} \int_{ \mathbb B_r}|\nabla u|^2\, dx-\frac{2}{r}\int_{\mathbb B_r} v \nabla u\cdot x\, dx+2\int_{\partial \mathbb B_r}u_n ^2 \, d\sigma,
\label{ddd}
\end{equation}
where $u_n=\frac{\partial u}{\partial n}=\nabla u\cdot \vec{n}$.
Performing similar calculations also shows that
\begin{align}
D_2^{'}(r)&=\frac{n-2}{r} \int_{ \mathbb B_r}|\nabla v|^2\, dx-\frac{2}{r}\int_{\mathbb B_r}\triangle v\nabla v\cdot x\, dx+\frac{2}{r^2}\int_{\partial \mathbb B_r}(\nabla v \cdot x)^2 \, d\sigma \nonumber \\
&=\frac{n-2}{r} \int_{\mathbb B_r}|\nabla v|^2\, dx-\frac{2}{r}\int_{ \mathbb B_r} \bar W(x) u\nabla v\cdot x\, dx+2\int_{\partial \mathbb B_r}v_n ^2 \, d\sigma.
\label{ddd1}
\end{align}
Direct calculations lead to that
\begin{align}
D'_3(r)&=\int_{\partial \mathbb B_r}({1}+\bar W(x)) uv \, d\sigma.
\label{ddd2}
\end{align}
We write $H(r)$ as
$$ H(r)=r^{n-1} \int_{\partial \mathbb B_1} u^2(rs)+v^2(rs) \, d\sigma. $$
Computing $H(r)$ with respect to $r$ gives that
\begin{equation}
H'(r)=\frac{n-1}{r} H(r)+2\int_{\partial \mathbb B_r} u_n u+ v_n v  \, d\sigma.
\label{hhh}
\end{equation}
If multiplying the first equation in (\ref{system1}) by $u$ and the second equation in (\ref{system}) by $v$, using integration by parts arguments, one has
\begin{equation}
D(r)=\int_{\partial \mathbb B_r} u_n u+ v_n v  \, d\sigma.
\end{equation}
Cauchy-Schwartz inequality yields that
\begin{align}
D^2(r)&\leq (\int_{\partial \mathbb B_r} u^2+v^2 \, d\sigma) (\int_{\partial \mathbb B_r} u^2_n+v^2_n \, d\sigma) \nonumber \\
&\leq H(r)\int_{\partial \mathbb B_r} u^2_n+v^2_n\, d\sigma.
\end{align}
If (\ref{NNNKao}) holds, then
$$  D(r)\leq r \int_{\partial \mathbb B_r} u^2_n+v^2_n \, d\sigma.  $$

The combination of the inequalities (\ref{ddd}), (\ref{ddd1}) and (\ref{ddd2}) yields that
\begin{align}
D'(r)&=\frac{n-2}{r} (\int_{ \mathbb B_r}|\nabla u|^2+|\nabla v|^2 \, dx)-\frac{2}{r}(\int_{\mathbb B_r}  v\nabla u\cdot x\, dx +\bar W(x)u\nabla v\cdot x\,dx)\nonumber \\& +2\int_{\partial \mathbb B_r}(u_n^2+v_n^2) \, d\sigma
+\int_{\partial \mathbb B_r}(1+\bar W(x)) uv \, d\sigma \nonumber \\
&=\frac{n-2}{r} D(r)-\frac{n-2}{r}\int_{ \mathbb B_r} (1+\bar W(x)) uv \, dx
\nonumber \\& -\frac{2}{r}(\int_{\mathbb B_r} v\nabla u\cdot x\, dx +\bar W(x)u\nabla v\cdot x\,dx)+2\int_{\partial \mathbb B_r} (u_n^2+v_n^2) \, d\sigma  \nonumber \\
&
+\int_{\partial \mathbb B_r}(1+\bar W(x)) uv \, d\sigma.
\label{hand}
\end{align}
We investigate the terms on the right hand side of (\ref{hand}). Using Cauchy-Schwartz inequality, we can get the following
\begin{equation}
\frac{2}{r}\int_{\mathbb B_r} v\nabla u\cdot x\, dx  \leq C\big(\int_{\mathbb B_r} v^2\, dx +\int_{\mathbb B_r} |\nabla u|^2\, dx\big),
\label{hand1}
\end{equation}
\begin{equation}
\frac{2}{r}\int_{\mathbb B_r} \bar W(x)u\nabla v\cdot x\, dx\leq C\big(\int_{\mathbb B_r} u^2\, dx +\int_{\mathbb B_r} |\nabla v|^2\, dx\big).
\label{hand3}
\end{equation}
From Cauchy-Schwartz inequality and the inequality (\ref{hanlin}), we derive that
\begin{align}
\frac{n-2}{r}\int_{\mathbb B_r} (1+\bar W(x)) vu\, dx &\leq \frac{C}{r} \int_{\mathbb B_r} u^2+v^2 \, dx \nonumber\\
&\leq C\int_{\partial \mathbb B_r} u^2+v^2 \, d\sigma+Cr \int_{ \mathbb B_r}|\nabla u|^2+|\nabla v|^2 \, dx\nonumber\\
&\leq CH(r)+ Cr D(r).
\label{hand5}
\end{align}

Together with the estimates (\ref{hand})--(\ref{hand5}), we arrive at
\begin{align*}
D'(r)&\geq \frac{n-2}{r} D(r)-CH(r)-CD(r)+2\int_{\partial \mathbb B_r} (u_n^2+v_n^2) \, d\sigma.
\label{handover}
\end{align*}
Since (\ref{NNN}) holds, furthermore, we have
\begin{align}
D'(r)\geq \frac{n-2}{r} D(r)-CD(r)+2\int_{\partial \mathbb B_r} (u_n^2+v_n^2) \, d\sigma.
\end{align}
Recall that $H'(r)$ is given in (\ref{hhh}). We consider the derivative of $I(r)$ with respective to $r$. Taking the definition of $D(r), H(r), I(r)$ and estimates (\ref{hhh}), (\ref{handover}) into account, we obtain that
\begin{align}
(\ln I(r))'&=\frac{1}{r}+\frac{D'(r)}{D(r)}-\frac{H'(r)}{H(r)} \nonumber\\
&\geq \frac{1}{r}+\frac{n-2}{r}-C+\frac{2\int_{\partial \mathbb B_r} (u_n^2+v_n^2) \, d\sigma}{ D(r)}-\frac{n-1}{r}- \frac{2\int_{\partial \mathbb B_r} (u_n u+v_n v ) \, d\sigma}{ H(r)}\nonumber\\
&\geq \frac{2\int_{\partial \mathbb B_r} (u_n^2+v_n^2) \, d\sigma}{\int_{\partial \mathbb B_r} (u_n u+v_n v ) \, d\sigma}- \frac{2\int_{\partial \mathbb B_r} (u_n u+v_n v ) \, d\sigma}{
\int_{\partial \mathbb B_r} ( u^2+ v^2 ) \, d\sigma}-C.
\end{align}
By Cauchy-Schwartz inequality, it follows that
\begin{equation}
\frac{I'(r)}{I(r)}\geq -C.
\label{fremono}
\end{equation}
Hence $e^{Cr}I(r)$ is monotone increasing in the component $(a_i, \ b_i)$. Thus, in this decomposition,
\begin{align}
I(r_1)&\leq I(r_2) e^{C(r_2-r_1)} \nonumber \\
&\leq e^{C(r_0-r_1)} I(r_2)
\label{in}
\end{align}
for $a_i<r_1<r_2<b_i<r_0$.
If $r_1\not \in (a_i, \ b_i)$, from the definition of the set $\Gamma$,
\begin{equation}I(r_1)\leq C.  \label{not} \end{equation}
Together with (\ref{in}) and (\ref{not}), for any $r_1\in (0, \ r_0)$, we get that
\begin{equation}
I(r_1)\leq C+e^{C(r_2-r_1)} I(r_2).
\end{equation}
If $r_0$ is sufficiently small, we have
\begin{equation}
I(r_1)\leq C+(1+\epsilon)I(r_2)
\end{equation}
for $0<r_1<r_2<r_0$. Therefore, the proposition is arrived.
\end{proof}

Let's derive some properties for $H(r)$. Since
$$ H'(r)=\frac{n-1}{r} H(r)+2D(r),  $$
then
\begin{equation}
\frac{d}{dr}\ln \frac{H(r)}{r^{n-1}}=\frac{2I(r)}{r}.
\label{star}
\end{equation}
Integrating from $R$ to $2R$ gives that
\begin{equation}
H(2R)=2^{n-1}H(R)\exp\{\int_R^{2R}\frac{2I(r)}{r}\,dr\}
\end{equation}
for $a_i\leq 2R\leq b_i$. Thus, from (\ref{fremono}),
\begin{equation}
H(2R)\leq 2^{n-1}H(R) 4^{ CI(b_i)}.
\label{ddoub}
\end{equation}
From (\ref{esti1}), we learn that $ \frac{I(r)}{r}\geq -C$ for some positive constant $C$. From (\ref{star}), it is also true that the function
\begin{equation} \frac{e^{Cr} H(r)}{r^{n-1}} \ \mbox{ is increasing for}   \quad r\in (0, \ r_0).
\label{amono}
\end{equation}

Following from the arguments in \cite{Lo} for harmonic functions, we show some applications of the almost monotonicity results for second order elliptic systems with potential functions.
\begin{corollary} Let $\epsilon$ be a small constant. There exists $R>0$ such that
\begin{equation}
(\frac{r_2}{r_1})^{2(1+\epsilon)^{-1}I(r_1)-C_1}\leq \frac{H(r_2)}{H(r_1)}\leq (\frac{r_2}{r_1})^{2(1+\epsilon)I(r_2)+C_1}
\label{coro1}
\end{equation}
for $0<r_1<r_2<R$.
\label{coro2}
\end{corollary}
\begin{proof}
For $0<r_1<r_2<r_0$, the integration of (\ref{star}) from $r_1$ to $r_2$ gives that
$$ H(r_2)=H(r_1)(\frac{r_2}{r_1})^{n-1} \exp\{ 2\int_{r_1}^{r_2} \frac{I(r)}{r} \,dr\}.   $$
Using the almost monotonicity of the frequency function in Proposition \ref{monoto}, we have
\begin{equation}
(\frac{r_2}{r_1})^{n-1} e^{(2(1+\epsilon)^{-1} N(r_1)-C)\ln \frac{r_2}{r_1}}  \leq \frac{H(r_2)}{H(r_1)}\leq (\frac{r_2}{r_1})^{n-1} e^{(2(1+\epsilon)I(r_2)+C)\ln \frac{r_2}{r_1}},
\end{equation}
which implies the corollary.
\end{proof}
We define the doubling index as
\begin{equation}
{N}(\mathbb B_r)=\log_2 \frac{\sup_{2\mathbb B_r}|( u, v)|}{\sup_{\mathbb B_r}|( u, v)|},
\label{defined}
\end{equation}
where $\sup_{\mathbb B_r}|( u, v)|=\|u\|_{L^\infty(\mathbb B_r)}+\|v\|_{L^\infty(\mathbb B_r)}$. For a positive number $\rho$, $\rho\mathbb B$ is denoted
as the ball scaled by a factor $\rho>0$ with the same center as $\mathbb B$. ${N}(x, r)$ is the double index for $(u, v)$ on the ball $\mathbb B(x, r)$.

Assume that $0< \epsilon<\frac{1}{10^8}$.
By standard elliptic estimates, the first equation of (\ref{system1}) implies that
\begin{equation}
\|u\|_{L^\infty(\mathbb B_r)}\leq C\epsilon^{-\frac{n}{2}} r^{-\frac{n}{2}}\big( \|u\|_{L^2(\mathbb B_{(1+\epsilon)^2r})}+
\|v\|_{L^2(\mathbb B_{(1+\epsilon)^2r})} \big)
\end{equation}
and the second equation of (\ref{system1}) gives that
\begin{equation}
\|v\|_{L^\infty(\mathbb B_r)}\leq C \epsilon^{-\frac{n}{2}} r^{-\frac{n}{2}}\big( \|v\|_{L^2(\mathbb B_{(1+\epsilon)^2r})}+
\|u\|_{L^2(\mathbb B_{(1+\epsilon)^2r})}\big ).
\end{equation}
Thus,
\begin{equation}
\|(u, v)\|_{L^\infty(\mathbb B_r)}\leq C  \epsilon^{-\frac{n}{2}} r^{-\frac{n}{2}} \|(u, v)\|_{L^2(\mathbb B_{(1+\epsilon)r})}.
\label{infinity}
\end{equation}
It is obvious that
\begin{equation}
\|(u, v)\|_{L^2(\mathbb B_r)}\leq C r^{\frac{n}{2}} \|(u, v)\|_{L^\infty(\mathbb B_r)}.
\label{infinity1}
\end{equation}

 Next we obtain a lower bound for ${N}(\mathbb B_r)$.
 The fact that $\frac{e^{Cr} H(r)}{r^{n-1}}$ is increasing in (\ref{amono}) and the inequality (\ref{infinity}) leads to
\begin{align}
\|(u, v)\|_{L^\infty(\mathbb B_r)}^2&\leq C \epsilon^{-n} r^{-n} \int_0^{(1+\epsilon)^2 r} \int_{\partial\mathbb B_s} |(u, v)|^2 \, d\sigma d s \nonumber \\
& \leq  C\epsilon^{-n} r^{-n} e^{C(1+\epsilon)^2r} \frac{H((1+\epsilon)^2 r)}{[(1+\epsilon) r]^{n-1}} \int_0^{(1+\epsilon)^2 r}\frac{s^{n-1}}{e^{Cs}} \,d s
 \nonumber \\
&\leq  C\frac{\epsilon^{-n} H((1+\epsilon)^2r)}{ r^{n-1}}e^{C(1+\epsilon)^2r_0}   \nonumber \\
&\leq  C\frac{\epsilon^{-n} H((1+\epsilon)^2r)}{ r^{n-1}}.
\label{end1}
\end{align}
From (\ref{infinity1}), it holds that
$$ \|(u, v)\|_{L^\infty(\mathbb B_{2r})}^2\geq \frac{C}{r^n} \int_{2(1-\epsilon)r}^{2r} H(s)\, ds.  $$
Thanks to the monotonicity of $\frac{e^{Cr} H(r)}{r^{n-1}}$  again,
\begin{align}
\|(u, v)\|_{L^\infty(\mathbb B_{2r})}^2&\geq
 C r^{-n} e^{C(1-\epsilon)r} \frac{H(2(1-\epsilon) r)}{[2(1-\epsilon) r]^{n-1}} \int_{2(1-\epsilon)r}^{2 r}\frac{s^{n-1}}{e^{Cs}} \,d s \nonumber \\
 &\geq C \frac{\epsilon H(2(1-\epsilon)r)}{r^{n-1}} e^{-2Cr_0}r_0  \nonumber \\
 &\geq C \frac{\epsilon H(2(1-\epsilon)r)}{r^{n-1}}.
\label{end2}
\end{align}
Therefore, from (\ref{end1}) and (\ref{end2}), we have
\begin{align}
{N}(\mathbb B_r)&=\log_2 \frac{\sup_{2\mathbb B_r}|( u, v)|}{\sup_{\mathbb B_r}|( u, v)|} \nonumber \\
&\geq \frac{1}{2}\log_2\frac{ \epsilon^{n+1} H(2(1-\epsilon)r)}{C H((1+\epsilon)^2r)}.
\end{align}
The lower estimates in (\ref{coro1}) leads to
\begin{align}
{N}(\mathbb B_r)&\geq \frac{1}{2}\log_2\Big(\frac{\epsilon^{n+1}}{C} \big[\frac{2(1-\epsilon)}{(1+\epsilon)^2}\big]^{2(1+\epsilon)^{-1}I((1+\epsilon)^2r)-C}\Big) \nonumber \\
& \geq I((1+\epsilon_1)r)(1-20\epsilon_1)+C\log_2\epsilon_1,
\label{what}
\end{align}
where  $(1+\epsilon)^2=1+\epsilon_1$ with $\epsilon $ sufficiently small.
We can also find an upper bound of the double index in term of the frequency function. Using (\ref{end1}) and (\ref{end2}), we have
\begin{align}
{N}(\mathbb B_r)&=\log_2 \frac{\sup_{2\mathbb B_r}|( u, v)|}{\sup_{\mathbb B_r}|( u, v)|} \nonumber \\
&\leq \frac{1}{2} \log_2 \frac{C\epsilon^{-n} r^{\frac{1-n}{2}} H\big(2(1+\epsilon)^2r\big)}{\epsilon r^{\frac{1-n}{2}} H\big((1-\epsilon)r\big) }.
\end{align}
It is true from (\ref{coro1}) that
\begin{equation}
\frac{H(2(1+\epsilon)^2r)}{H((1-\epsilon)r)} \leq \big(\frac{2(1+\epsilon)^2}{1-\epsilon}\big)^{2(1+\epsilon)I(2(1+\epsilon)^2  r)+C_2}.
\end{equation}
Thus, we further obtain that
\begin{align}
{N}(\mathbb B_r) &\leq \frac{1}{2}\log_2 \Big(C \epsilon^{-n-1}\big(\frac{2(1+\epsilon)^2}{1-\epsilon}\big)^{2(1+\epsilon)^2 I(2(1+\epsilon)^2r)+C_2}\Big).
\label{what1}
\end{align}
Let $(1+\epsilon)^2=1+\epsilon_1$ again, i.e. $\epsilon_1\approx 2\epsilon$. We can check that
$$ \frac{1}{2}\log_2 \Big(C \epsilon^{-n-1}\big[\frac{2(1+\epsilon)^2}{1-\epsilon}\big]^{2(1+\epsilon)^2 I(2(1+\epsilon)^2r)+C_2}\Big)
\leq I(2r(1+\epsilon_1))(1+20 \epsilon_1)-C\log_2\epsilon_1
$$
for $\epsilon $ sufficiently small.

In conclusion, from (\ref{what}) and (\ref{what1}), we have shown that
\begin{equation}
I( r(1+\epsilon_1))(1-50\epsilon_1)+C\log_2\epsilon_1 \leq N(\mathbb B_r)\leq I(  2r(1+\epsilon_1))(1+50\epsilon_1)-C\log_2\epsilon_1.
\label{compare}
\end{equation}

\begin{lemma}
Let $\epsilon$ be a small positive constant. There exists $R$ such that
\begin{equation}
t^{N(x,\rho)(1-\epsilon)+C\log_2\epsilon} \sup_{\mathbb B_{\rho}(x)}|(u, v)|\leq \sup_{\mathbb B_{t\rho}(x)}|(u, v)| \leq t^{N(x,t\rho)(1+\epsilon)-C\log_2\epsilon} \sup_{\mathbb B_{\rho}(x)}|(u, v)|
\label{overrr}
\end{equation}
for $t>2$, $t\rho<R$ and any $x\in \mathbb B_R$ with $\mathbb B_{t\rho}(x)\subset \mathbb B_R$.
Furthermore, there exists $N_0$ such that if $N(x, \rho)\geq N_0$, then
\begin{equation}
t^{N(x,\rho)(1-\epsilon)} \sup_{\mathbb B_{\rho}(x)}|(u, v)|\leq \sup_{\mathbb B_{t\rho}(x)}|(u, v)|\leq  t^{N(x,t\rho)(1+\epsilon)} \sup_{\mathbb B_{\rho}(x)}|(u, v)| .
\label{need1}
\end{equation}
\label{cann}
\end{lemma}
\begin{proof}
We first show the proof of the left hand side of (\ref{overrr}) and (\ref{need1}). We assume $t> 2^{1+\epsilon}$. If not, then $2<t\leq  2^{1+\epsilon}$. It follows that
\begin{align*}
\sup_{\mathbb B_{t\rho}(x)}|(u, v)|&\geq \sup_{\mathbb B_{2\rho}(x)}|(u, v)|\geq 2^{N(x,\rho)} \sup_{\mathbb B_{\rho}(x)}|(u, v)|\nonumber \\
&\geq t^{N(x,\rho)(1-\epsilon)} \sup_{\mathbb B_{\rho}(x)}|(u, v)|,
\end{align*}
since $2>t^{1-\epsilon}$ in this case. Then the left hand side of (\ref{overrr}) is shown.

Now we consider $t> 2^{1+\epsilon}$. It is true that
\begin{equation}
\frac{H(x, t \rho)}{(t\rho)^{n-1}}\leq \sup_{\mathbb B_{t\rho}(x)}|(u, v)|^2.
\label{conf1}
\end{equation}
Choose $\epsilon_1=\frac{\epsilon}{500}$. Applying (\ref{compare}) by considering the doubling index in $\mathbb B_\rho(x)$, we obtain that
\begin{equation}
I(2\rho(1+\epsilon_1))\geq \frac{N(x,\rho)+C\log_2\epsilon_1}{1+50\epsilon_1}.
\end{equation}
From the monotonicity of $H(r)$ in Corollary \ref{coro2} and last inequality, we get that
\begin{equation}
H(x, t\rho)\geq H(x, 2\rho(1+\epsilon_1))\Big(\frac{t}{2(1+\epsilon_1)}\Big)^{\frac{2N(x,\rho)}{(1+50\epsilon_1)(1+\epsilon_1)}+C\log_2\epsilon_1}.
\label{conf2}
\end{equation}
 Note that $t> 2^{1+\epsilon}$ implies that $t>2(1+\epsilon_1)$. Furthermore, the estimates (\ref{end1}) and the definition of the doubling index  yield that
\begin{align}
H(x, 2\rho(1+\epsilon_1))&\geq C\epsilon_1^n \rho^{n-1}\sup_{\mathbb B_{2\rho}(x)}|(u, v)|^2 \nonumber\\
&=C2^{2N(x,\rho)}\epsilon_1^n \rho^{n-1} \sup_{\mathbb B_{\rho}(x)} |(u, v)|^2.
\label{conf3}
\end{align}

In view of (\ref{conf1}), (\ref{conf2}) and (\ref{conf3}), we arrive at
\begin{equation}
\sup_{\mathbb B_{t\rho}(x)} |(u, v)|\geq C 2^{N(x,\rho)}\epsilon_1^n t^{-\frac{n-1}{2}} \Big(\frac{t}{2(1+\epsilon_1)}\Big)^{\frac{N(x,\rho)}{(1+100\epsilon_1)}+C\log_2\epsilon_1}\sup_{\mathbb B_{\rho}(x)} |(u, v)|.
\label{lei1}
\end{equation}
Note that
\begin{align}
t^{\frac{N(x,\rho)}{(1+100\epsilon_1)}} &\geq t^{N(x,\rho)(1-\frac{\epsilon}{2})}\geq t^{N(x,\rho)(1-\epsilon)} 2^{\frac{N(x,\rho)}{2}\epsilon   } \nonumber \\
&\geq t^{N(x,\rho)(1-\epsilon)} (1+\epsilon_1)^{\frac{N(x,\rho)}{1+100\epsilon_1}},
\label{lei2}
\end{align}
since $ 2^{{\epsilon}/{2}}\geq 1+\frac{\epsilon}{50}$.

Notice that $\epsilon_1^n> t^{C\log_2\epsilon_1}$. From (\ref{lei1}) and (\ref{lei2}), we deduce that
\begin{equation}
\sup_{\mathbb B_{t\rho}(x)}|(u, v)|\geq t^{N(x,\rho)(1-\epsilon)+C\log_2\epsilon} \sup_{\mathbb B_{\rho}(x)}|(u, v)|.
\end{equation}
We may choose a smaller $\epsilon$ so that
\begin{equation}
\sup_{\mathbb B_{t\rho}(x)}|(u, v)|\geq t^{N(x,\rho)(1-2\epsilon)+C\log_2(2\epsilon)} \sup_{\mathbb B_{\rho}(x)}|(u, v)|.
\end{equation}
Let $N_0=\frac{C\log_2(2\epsilon)}{\epsilon}$. Furthermore, if $N> N_0$, we get that
$$\sup_{\mathbb B_{t\rho}(x)}|(u, v)|\geq t^{N(x,\rho)(1-\epsilon)} \sup_{\mathbb B_{\rho}(x)}|(u, v)|.$$
This completes the proof of left hand side of (\ref{overrr}) and (\ref{need1}).

By the similar strategy, we can show that there exists $R$ such that
\begin{equation}
\sup_{\mathbb B_{t\rho}(x)}|(u, v)|\leq t^{N(x,t\rho)(1+\epsilon)-C\log_2\epsilon} \sup_{\mathbb B_{\rho}(x)}|(u, v)|
\label{overr}
\end{equation}
for $t\rho<R$ and any $x\in \mathbb B_R$ with $\mathbb B_{t\rho}(x)\subset \mathbb B_R$.
Furthermore, there exists $N_0$ such that if $N(x, \rho)\geq N_0$, then
\begin{equation}
\sup_{\mathbb B_{t\rho}(x)}|(u, v)|\leq t^{N(x,t\rho)(1+\epsilon)} \sup_{\mathbb B_{\rho}(x)}|(u, v)|.
\label{need2}
\end{equation}
Thus, we arrive at the right hand side of (\ref{overrr}) and (\ref{need1}). Therefore,
the proof of the lemma is completed.
\end{proof}

Proceeding  as the argument in \cite{Lo} and using (\ref{need1}), we can compare doubling index at nearby points.
\begin{lemma}
There exist $R$ and $N_0$ such that for any points $x_1, x_2 \in \mathbb B_r$ and $\rho$ such that $N(x_1, \rho)>N_0$ and $d(x_1, x_2)<\rho<R$
, there exists $C$ such that
\begin{equation}
N(x_2, C\rho)>\frac{99}{100} N(x_1, \rho).
\end{equation}
\label{lemcom}
\end{lemma}

\section{Nodal sets of bi-Laplace equations }

Let $n\geq 3$ in this section. After those preparations, we follow the new combinatorial argument in the seminal work of \cite{Lo} in this section.
Let $x_1, x_2, \cdots, x_{n+1}$ be the vertices of a simplex $S$ in $\mathbb R^n$. Denote $\diam(S)$ as the diameter of the simplex $S$. We use $width(S)$ to denote the minimum distance between two parallel hyperplanes that contain $S$. The symbol $w(S)$ is defined as the relative width of $S$:
$$ w(S)=\frac{width(S)}{diam(S)}. $$
We assume $w(S)>\gamma$ for some constant $\gamma$. In particular, $x_1, x_2, \cdots, x_{n+1}$ are assumed not to be on the same hyperplane. We denote $x_0$ as the barycenter of $S$, i.e. $x_0=\frac{1}{n+1}\sum^{n+1}_{i=1} x_i$. Roughly speaking, next lemma shows that the doubling index will accumulate at the barycenter of the simplex if the doubling index at the vertices $\{x_1, x_2, \cdots, x_{n+1}\}$ are large. Using the Lemma \ref{cann} on frequency function and Logunov's arugment. The following lemma holds.
\begin{lemma}
Let $\mathbb B_i$ be balls centered at $x_i$ with radius less than $\frac{K diam(S)}{2}$ for some $K$ depending only on $\gamma$, $i=1, 2, \cdots, n+1$. There exist positive constants $c=c(\gamma, n)$, $C=C(\gamma, n)\geq K$, $r=r(\gamma)$ and $N_0=N_0(\gamma)$ such that if $S\subset \mathbb B_r$ and $N(\mathbb B_i)>N$ with $N>N_0$ for each $i$, $i=1, 2, \cdots, n+1$, then
$$N(x_0, C\diam(S))>(1+c)N.  $$
\label{simplex}
\end{lemma}

We introduce the doubling index of the cube $Q$. For a given cube $Q$, define the doubling index $N(Q)$ as
$$ N(Q)=\sup_{x\in Q, \ r\in(0, diam(Q))} N(x, r). $$
The doubling index of the cube $N(Q)$ is more convenient in applications. Obviously, if a cube $q\subset Q$, then $N(q)\leq N(Q)$.
If a cube $q\subset \cup_i Q_i$ with $\diam(Q_i)\geq \diam q$, then $N(Q_i)\geq N(q)$ for some $Q_i$.

Based on the propagation of smallness of the Cauchy data in Lemma \ref{halfspace} and the arguments in \cite{Lo} or \cite{GR}, for the completeness of presentation, we can show the following lemma. Roughly speaking, it asserts that if a set of sub-cubes
with intersection with a hyperplane all have a large doubling index, then the original cube that contains those sub-cubes at least have double doubling index.
\begin{lemma}
Let $Q$ be a cube $[-R, \ R]^n$ in $\mathbb R^n$. Divide $Q$ into $(2A+1)^n$ equal subcubes $q_i$ with side length $\frac{2R}{2A+1}$. Let $\{q_{i, 0}\}$ be the subcubes with nonempty intersection with the hyperplane $\{x_n=0\}$. For each $q_{i, 0}$, there exist some point $x_i \in q_{i, 0}$ and $r_i<10 \diam(q_{i, 0})$ such that $N(x_i, r_i)>N$, where $N$ is a large fixed number. The following property holds:
If $A>A_0$, $R<R_0$ and $N>N_0$ for some $A_0$, $R_0$ and $N_0$, then
$$N(Q)\geq 2N.   $$
\label{lemdou}
\end{lemma}
\begin{proof}
By scaling, we may assume that $R=\frac{1}{2}$ and $R_0>\frac{1}{2}$. Let $\mathbb B$ be the unit ball.
Let $\sup_{\frac{1}{4}\mathbb B}|(u, v)|=M_0$, we have $$\sup_{\mathbb B_{1/8}(x_i)}|(u, v)|\leq M_0$$ if $x_i\in \frac{1}{8}\mathbb B$ since
$ \mathbb B_{1/8}(x_i)\subset \frac{1}{4}\mathbb B$. From the assumption $N(x_i, r_i)\geq N$ and doubling lemma, we get
\begin{align}
\sup_{  4 q_{i, 0}}|(u, v)|&\leq \sup_{\mathbb B_\frac{ 8\sqrt{n}}{2A+1}(x_i)} |(u, v)|\leq C \sup_{\mathbb B_{1/8}(x_i)}|(u, v)|(\frac{64\sqrt{n}}{2A+1})^{\frac{N}{2}} \nonumber \\
&\leq 2^{-CN\log A} M_0,
\label{useback}
\end{align}
where the constants $N$ and $A$ are assumed to be large.
The following interpolation inequality is known, e.g. \cite{BL},
\begin{align}
\|\nabla f\|_{L^2(\mathbb R^{n-1})}\leq C( \|f\|_{W^{2,2}(\mathbb R^n)}+ \|f\|_{L^2(\mathbb R^{n-1})})
\label{interpo}
\end{align}
for any $f\in W^{2,2}(\mathbb R^n)$. By replacing $f$ by $\psi (u, v)$, where $\psi$ is a smooth cut-off function with $\psi=1$ in $\mathbb B_r$ and $\psi=0$ outside $\mathbb B_{2r}$, we obtain that
\begin{align}
\|\nabla (u, v)\|_{L^2(\mathbb R^{n-1}\cap\mathbb B_r)}\leq C\big(\| (u, v)\|_{W^{2,2}(\mathbb B_{2r})}+ \|(u, v)\|_{L^2(\mathbb R^{n-1}\cap \mathbb B_{2r})}\big).
\label{lastine}
\end{align}
Let $\tilde{\Gamma}=\frac{1}{8} \mathbb B\cap \{x_n=0\}$. The last inequality, trace inequalities and elliptic estimates yield that
\begin{align}
\|\nabla(u, v)\|_{L^2(\tilde{\Gamma}\cap q_{i,0})}&\leq C(2A+1)^2( \|(u, v)\|_{W^{2,2}(2q_{i,0})}+ \|(u, v)\|_{L^2(\tilde{\Gamma}\cap 2q_{i,0})} ) \nonumber \\
&\leq C(2A+1)^4 \|(u, v)\|_{L^2(4q_{i,0})}.
\end{align}
Using the trace inequality and elliptic estimates again, we obtain that
\begin{align}
\|(u, v)\|_{W^{1,2}(\tilde{\Gamma}\cap q_{i,0})}+\|\frac{\partial(u, v)}{\partial n}\|_{L^2(\tilde{\Gamma}\cap q_{i,0})}&\leq C(2A+1)
\|(u, v)\|_{W^{2,2}(\tilde{\Gamma}\cap 2q_{i,0})}+ \|\nabla (u, v)\|_{L^2(\tilde{\Gamma}\cap q_{i,0})} \nonumber \\
&\leq C (2A+1)^4 \| (u, v)\|_{L^2(3 q_{i,0})} \nonumber \\
&\leq C\frac{ (2A+1)^4 }{ (2A+1)^\frac{n}{2}}  \| (u, v)\|_{L^\infty(4 q_{i,0})}.
\end{align}
Summing up all the cubes $q_{i,0}$ with intersection with $\tilde{\Gamma}$, the last inequality yields that
\begin{align}
\|(u, v)\|_{W^{1,2}(\tilde{\Gamma})}+\|\frac{\partial(u, v)}{\partial n}\|_{L^2(\tilde{\Gamma})}&\leq C (2A+1)^{\frac{n}{2}+3}\| (u, v)\|_{L^\infty(4 q_{i,0})}  \nonumber \\
&\leq e^{-CN\log A}M_0,
\end{align}
where we used (\ref{useback}) in the second inequality. Note that $\|(u,v)\|_{L^2( \frac{1}{4}\mathbb B^+)}\leq CM_0$. By scaling and using the propagation of smallness Lemma \ref{halfspace} in Section 7, we have
\begin{align}
\|(u,v)\|_{L^2( 2^{-10}\mathbb B^+)}\leq  e^{-CN\log A} M_0.
\end{align}
We select a ball $\mathbb B_{2^{-11}}(p)\subset {2^{-10}}\mathbb B^+$.  Thus, by elliptic estimates,
\begin{align}
\|(u,v)\|_{L^\infty( \mathbb B_{2^{-12}}(p))}\leq e^{-CN\log A} M_0.
\end{align}
By the fact that $\|(u, v)\|_{L^\infty (\mathbb B_{\frac{1}{4}}(p))}\geq M_0$, we derive that
\begin{align}
\frac{ \|(u, v)\|_{L^\infty (\mathbb B_{\frac{1}{4}}(p))}} {\|(u,v)\|_{L^\infty( \mathbb B_{2^{-12}}(p))}}\geq e^{CN \log A}.
\end{align}
The doubling lemma gives that
\begin{align}
\frac{ \|(u, v)\|_{L^\infty (\mathbb B_{\frac{1}{4}}(p))}} {\|(u,v)\|_{L^\infty( \mathbb B_{2^{-12}}(p))}  }\leq  (2^{10})^{\tilde{N}},
\end{align}
where $\tilde{N}$ is the doubling index in $\mathbb B_{\frac{1}{4}}(p)$. Therefore,
\begin{align}
\tilde{N}\geq 2N
\end{align}
if $A$ is large enough.
\end{proof}

Following the arguments in \cite{Lo}, the following lemma holds.
\begin{lemma}
If $Q$ is partitioned into $A^n$ equal sub-cubes, where $A$ depends on $n$, then the number of sub-cubes with doubling index greater than
$\max\{ \frac{N(Q)}{1+c}, \  N_0\}$ is less than $\frac{1}{2}A^{n-1}$ for some $c$ depending $n$ and some fixed constant $N_0$.
\label{lem12}
\end{lemma}

Now we give the estimates of the nodal set $\{u=v=0\}$ for the elliptic system (\ref{system1}) in a small cube. We show the details of the following proposition.
\begin{proposition}
Let $N_{(u, v)}(Q)$ be the doubling index of the cube $Q$ for the solutions $(u, v)$ in (\ref{system1}). There exist positive constant $r$, $C$ and ${\hat{\alpha}}$ such that for any solutions $(u, v)$ on $\mathcal{M}$ and $Q\subset \mathbb B_r$,
\begin{equation}
H^{n-1}(\{u=v=0\}\cap Q)\leq C diam^{n-1}(Q)N_{(u, v)}^{\hat{\alpha}}(Q),
\label{prove}
\end{equation}
where ${\hat{\alpha}}$ depends only on $n$ and $N_{(u, v)}(Q)$ is the doubling index on $Q$ for the function $(u, v)$.
\label{pro2}
\end{proposition}

\begin{proof}
Let the cube $Q\subset \mathbb B_r$. For any solutions $(u, v)$ in the elliptic systems (\ref{system1}), we consider those solutions such that $N_{(u, v)}(Q)\leq N$.
Define the function
\begin{equation} F(N)=\sup_{N_{(u, v)}(Q)\leq N}\frac{ H^{n-1}(\{u=v=0\}\cap Q)}{ \diam^{n-1}(Q)}.
\label{defnew} \end{equation}
We are going to show that
$$  F(N)\leq C N^{\hat{\alpha}} $$
for some $\hat{\alpha}$ depending only on $n$,  which provides the proof of the proposition. As shown in \cite{Han} for higher order elliptic equations, the Hausdorff dimension of the  sets $\{D^\nu u=0 \ \mbox{for all} \ |\nu|\leq 2\}$ is not greater than $n-1$. Since $v=\triangle u$, the mixed nodal sets $\{ Q| u=v=0 \}$ is not greater than $n-1$. The Hausdorff dimension of nodal sets  $\{ u=0 \}$ is no more than $n-1$. Such stratification can also be observed in Lemma \ref{haudim} in Section 6. Obviously, the mixed nodal sets $\{ u=v=0 \}$ is subset of the nodal sets $\{u=0\}$. Even if there exists co-dimension one nodal sets $\{u=0\}$ in $Q$, it does not guarantee the existence of co-dimension one mixed nodal sets $\{u=v=0\}$.
 We assume that $u$ and $v$ has the same co-dimension one zero sets in $Q$. Otherwise, $H^{n-1}(\{u=v=0\}\cap Q)=0$, then the proposition follows immediately. If there exist $x_0$ such that $u(x_0)=v(x_0)=0$ in $Q$, then $N_{(u, v)}(Q)\geq 1$. In Lemma \ref{bounf} in the Appendix, we have shown  that $F(N)< \infty$. We claim that if
\begin{equation} F(N)> 3AF(\frac{N}{1+c}),
\label{claim}
\end{equation}
then the set $N\leq N_0$, where the constant $A$, $c$ are those in the last lemma and $N_0$ depends on the manifold $\mathcal{M}$. If  $F(N)$ is almost attained in (\ref{defnew}), then
\begin{equation}
\frac{ H^{n-1}(\{u=v=0\}\cap Q)}{ \diam^{n-1}(Q)}>\frac{5}{6} F(N),
\label{contra}
\end{equation}
where  $N_{(u, v)}(Q)\leq N$. We divide $Q$ into $A^n$ equal subcubes $q_i$, $i=1, 2, \cdots, A^n$, then split $q_i$ into two groups
$$G_1= \{ q_i|\frac{N}{1+c}\leq N(q_i)\leq N\}   $$
and
$$ G_2=\{ q_i|N(q_i)<\frac{N}{1+c} \}.   $$
Thanks to the Lemma \ref{lem12} , we know that the number of subcubes in $G_1$ less than $\frac{1}{2}A^{n-1}$ if $N>N_0$.  We have
\begin{align}
H^{n-1}(\{u=v=0\}\cap Q)&\leq \sum_{q_i\in G_1} H^{n-1}(\{u=v=0\}\cap q_i)+\sum_{q_i\in G_2} H^{n-1}(\{u=v=0\}\cap q_i) \nonumber \\
&\leq |G_1|F(N) \frac{ \diam^{n-1}(Q)}{A^{n-1}}+|G_2| F(\frac{N}{1+c})  \frac{ \diam^{n-1}(Q)}{A^{n-1}} \nonumber \\
&=I_1+I_2,
\end{align}
where $|G_i|$ denotes the number of subcubes in $G_i$.  Since $|G_1|\leq \frac{1}{2} A^{n-1}$, then
\begin{equation}
I_1\leq \frac{1}{2} F(N) \diam^{n-1}(Q).
\label{II1}
\end{equation}
Since (\ref{claim}) holds, it follows that
\begin{equation}
I_2\leq |G_2|\frac{F(N)}{3A} \frac{ \diam^{n-1}(Q)}{A^{n-1}}.
\end{equation}
It is obvious that $|G_2|\leq A^n$. Then
\begin{equation}
I_2\leq \frac{1}{3}F(N)\diam^{n-1}(Q).
\label{II2}
\end{equation}
The combination of (\ref{II1}) and (\ref{II2}) yields that
\begin{equation}
I_1+I_2\leq \frac{5}{6} F(N)\diam^{n-1}(Q),
\end{equation}
which is a contradiction to (\ref{contra}). Therefore, we have shown the claim. That is, if the set $N\geq N_0$, then
\begin{equation} F(N)\leq 3AF(\frac{N}{1+c}).
\label{result}
\end{equation}
Let $\frac{N}{(1+c)^m}=N_0$. We iterate the estimate (\ref{result}) $m$ times to get
\begin{align*}
F(N)&\leq (3A)^m F(\frac{N}{(1+c)^m}) \\
&=(1+c)^{(\log_{1+c} {3A}) (\log_{1+c} \frac{N}{N_0})} F(N_0)\\
&=(\frac{N}{N_0})^{(\log_{1+c} 3A)}F(N_0).
\end{align*}
Thus, we show the conclusion (\ref{prove}) for $N\geq N_0$. If $N\leq N_0$, by the Lemma \ref{bounf}, we obtain that
\begin{equation}
F(N)\leq C(N_0)
\end{equation}
for some $C$ that depends on $N_0$. Therefore,  the proposition is completed.

\end{proof}
With the aid of the upper bound of nodal sets in a small cube in the proposition, we provide the proof of Theorem \ref{th2} for bi-Laplace equations (\ref{bi-Laplace}).
\begin{proof}[ Proof of Theorem \ref{th2}]
By the elliptic regularity estimates, from the doubling inequality in Theorem \ref{th1}, we have the following $L^\infty$ type doubling inequality,
\begin{equation}
\|(u, \triangle u)\|_{L^\infty(\mathbb B_{2r}(x))}\leq e^{ CM^{\frac{2}{3}}}\|(u, \triangle u) \|_{L^\infty(\mathbb B_{r}(x))}
\end{equation}
for any $x\in \mathcal{M}$ and any $0<r<r_0$, where $r_0$ depends only on the manifold $\mathcal{M}$.
From the definition of doubling index in (\ref{defined}) , we know that
$$N(x, r)\leq CM^{\frac{2}{3}}$$
for $M$ sufficiently large and for any $x\in \mathcal{M}$ and $0<r<r_0$. Thus, the doubling index $N(Q)\leq CM^{\frac{2}{3}}$ in the cube $Q$. We consider the reduced elliptic systems (\ref{system1}) of bi-Laplace equations (\ref{bi-Laplace}) in the cube $Q\subset \mathbb B_r$ with $0<r<\frac{r_0}{{M}^{1/4}}$. Note that $v=\triangle u$. From the last proposition, we get
$$ H^{n-1}(\{u=\triangle u =0\}\cap Q)\leq C  M^{\frac{2\hat{\alpha}}{3}-\frac{n-1}{4}}.  $$
Since the manifold $\mathcal{M}$ is compact, we can cover the manifold with $CM^\frac{n}{4}$ number of balls $\mathbb B_r$ with $0<r<\frac{r_0}{{M}^{1/4}}$. Therefore, we arrive at
$$ H^{n-1}(\{u=\triangle u =0\})\leq C  M^{\frac{2\hat{\alpha}}{3}+\frac{1}{4}}.  $$
This gives the proof of Theorem \ref{th2}.
\end{proof}

\begin{remark}
For the 2-dimensional compact smooth manifolds, a polynomial upper bound with explicit power $\alpha$ for the mixed nodal sets $\{x\in\mathcal{M}|u=\triangle u=0\}$ of solutions of bi-Laplace equations \ref{bi-Laplace} might be obtained using the ideas in \cite{DF1} and \cite{LM}. The author hopes to explore it in the future work.
\end{remark}

\section{Carleman estimates }

In this section, we show the doubling inequalities for the bi-Laplace equations (\ref{bi-Laplace}). We use Carleman estimates
to obtain
 some quantitative
type of Hadamard's three balls theorem, then employ the ``propagation of
smallness" argument to get some lower bound of $L^2$ norm of solutions in a small ball. At last, using Carleman estimates and the lower bound of $L^2$ norm of solutions, we obtain the uniform doubling inequality.

For any $x_0\in \mathcal{M}$, let $r=d(x, \ x_0)=r(x)$ be the Riemannian distance from $x_0$ to $x$. $\mathbb B_r(x_0)$ is denoted as the
geodesic ball at $x_0$ with radius $r$. The symbol $\|\cdot\|$ denotes the $L^2$ norm. Specifically, $\|\cdot\|_{\mathbb B_r(x_0)}$ or $\|\cdot\|_r$ for short denotes the $L^2$ norm on the ball $\mathbb B_r(x_0)$.
Our crucial tools to get the doubling inequality are the quantitative
 Carleman estimates. Carleman estimates are weighted integral inequalities with a weight function $e^{\tau\phi}$, where $\phi$ usually satisfies some convex condition. We construct the weight  function $\phi$ as follows.
 Set $$\phi=-g(\ln r(x)),$$ where $g(t)=t-e^{\epsilon t}$ for some small $0<\epsilon<1$ and $-\infty<t <T_0$. The positive constant $\epsilon$ is a fixed small number and $T_0$ is negative with $|T_0|$ large enough. One can check that
 \begin{equation}
 \lim_{t \to -\infty}-e^{-t} g''(t)=\infty \quad \mbox{and} \quad \lim_{t \to -\infty} g'(t)=1.
 \end{equation}
 Such weight function $\phi$ was introduced by H\"ormander in \cite{H}.
 The following Carleman estimates are shown in \cite{Bak}.
There exist positive constant $R_0$, $C$, which depend only on the manifold $\mathcal{M}$ and $\epsilon$, such that, for any $x_0\in  \mathcal{M}$, any $f\in C^\infty_0(\mathbb B_{R_0}(x_0)\backslash \mathbb B_{\delta}(x_0))$ with $0<\delta<R_0$, and $\tau>C$,
 the following Carleman estimate holds
 \begin{align}
 C\| r^2 e^{\tau \phi} \triangle f\|&\geq \tau^{\frac{3}{2}}\|r^{\frac{\epsilon}{2}}e^{\tau\phi} f\|+\tau \delta \|r^{-1} e^{\tau\phi} f\|  \nonumber \\ &+
 \tau^{\frac{1}{2}}\|r^{1+\frac{\epsilon}{2}}e^{\tau\phi} \nabla f\|.
 \label{similar}
\end{align}
Similar type of Carleman estimates without the second term on the right hand side of (\ref{similar}) are well-known in the literature, see e.g. \cite{AKS},  \cite{EV}, \cite{H1}, \cite{K}, to just mention a few. There has been a long and rich history for the development of Carleman estimates. It is hard to provide an exhaustive list for the applications of such estimates. Interested reads may refer to those literature or references therein for more history about such $L^2$ type Carleman estimates.
The Carleman estimates (\ref{similar}) also hold for vector functions. Let $F=(f_1, f_2)$. If $F\in C^\infty_0(\mathbb B_{R_0}(x_0)\backslash \mathbb B_{\delta}(x_0), \mathbb R^2)$, similar arguments as (\ref{similar}) show that
 \begin{align}
 C\| r^2 e^{\tau \phi} \triangle F\|&\geq \tau^{\frac{3}{2}}\|r^{\frac{\epsilon}{2}}e^{\tau\phi} F\|+\tau \delta \|r^{-1} e^{\tau\phi} F\| \nonumber \\ &+
  \tau^{\frac{1}{2}}\|r^{1+\frac{\epsilon}{2}}e^{\tau\phi} \nabla F\|.
 \label{vec}
 \end{align}

 Let $V(x)=\begin{pmatrix}
0, & 1 \\
W(x), & 0
\end{pmatrix}
$.
Following from (\ref{vec}), we can easily establish  the quantitative Carleman estimates for vector functions.
\begin{lemma}
There exist positive constants $R_0$, $C$, which depend only on the manifold $\mathcal{M}$ and $\epsilon$, such that, for any $x_0\in  \mathcal{M}$, $F\in C^\infty_0(\mathbb B_{R_0}(x_0)\backslash \mathbb B_{\delta}(x_0), \mathbb R^2)$ with $0<\delta<R_0$,  and $\tau>C(1+\|V\|_{L^\infty}^{\frac{2}{3}})$, one has
  \begin{align}
 C\| r^2 e^{\tau \phi} (\triangle F-V(x,y)\cdot F)\|&\geq \tau^{\frac{3}{2}}\|r^{\frac{\epsilon}{2}}e^{\tau\phi} F\|+\tau \delta \|r^{-1} e^{\tau\phi} F\| \nonumber \\ &+
  \tau^{\frac{1}{2}}\|r^{1+\frac{\epsilon}{2}}e^{\tau\phi} \nabla F\|.
 \label{Carle}
 \end{align}
 \label{lemma1}
\end{lemma}
\begin{proof}
By triangle inequality and the inequality (\ref{vec}), we have
\begin{align}
C\| r^2 e^{\tau \phi} (\triangle F-V(x,y)\cdot F)\|& \geq C\|r^2 e^{\tau \phi} \triangle F\|- C\|r^2 e^{\tau \phi} V(x,y)\cdot F\|\nonumber \\
&\geq  \tau^{\frac{3}{2}}\|r^{\frac{\epsilon}{2}}e^{\tau\phi} F\|+\tau \delta \|r^{-1} e^{\tau\phi} F\| \nonumber \\ &+
  \tau^{\frac{1}{2}}\|r^{1+\frac{\epsilon}{2}}e^{\tau\phi} \nabla F\|-C \|V\|_{L^\infty}\|r^2 e^{\tau \phi}  F\|.
\end{align}
If $\tau$ is chosen to be so large that $\tau^{\frac{3}{2}}\geq C(1+\|V\|_{L^\infty})$, the estimates (\ref{Carle}) can be derived.
\end{proof}

To show the refined doubling inequality in Theorem \ref{th3}, we establish the following Carleman estimates for the bi-Laplace operator involving the potential $W(x)$.
\begin{lemma}
There exist positive constants $R_0$, $C$, which depend only on the manifold $\mathcal{M}$ and $\epsilon$, such that, for any $x_0\in  \mathcal{M}$, any $f\in C^\infty_0(\mathbb B_{R_0}(x_0)\backslash \mathbb B_{\delta}(x_0))$ with $0<\delta<R_0$,  and $\tau>C(1+\|W\|_{L^\infty}^\frac{1}{3})$, one has
 \begin{align}
 C\| r^4 e^{\tau \phi} (\triangle^2 f-W(x)f)\|\geq \tau^3\|r^{\epsilon}e^{\tau\phi} f\|+\tau^2 \delta^2 \|r^{-2} e^{\tau\phi} f\|.
 \label{Carle2}
 \end{align}
 \label{lemmm}
\end{lemma}

\begin{proof}
Recall the definition of the weight function $\phi=-\ln r+r^\epsilon$, we see that
$$  r^4 e^{\tau\phi}=r^2 e^{(\tau-2)\phi}e^{2 r^\epsilon}. $$
Since $0<r<R_0<1$, then $1<e^{2 r^\epsilon}<e^2$. It follows from (\ref{similar}) that
\begin{align}
C^2\| r^4 e^{\tau \phi} \triangle^2 f \|&\geq  C\| r^2 e^{(\tau-2) \phi} \triangle^2 f \| \nonumber\\
&\geq \tau^{\frac{3}{2}}\|r^{\frac{\epsilon}{2}} e^{(\tau-2)\phi} \triangle f\|.
\label{hih1}
\end{align}
Since
\begin{align}r^{\frac{\epsilon}{2}}e^{(\tau-2)\phi}&=r^2 e^{\tau\phi}r^{\frac{\epsilon}{2}} e^{-2 r^\epsilon}\nonumber \\
&= r^2 e^{ (\tau-\frac{\epsilon}{2})\phi  } e^{\frac{\epsilon}{2} r^\epsilon} e^{-2 r^\epsilon}, \end{align}
it follows that
$$|r^{\frac{\epsilon}{2}}e^{(\tau-2)\phi}|\geq C r^2  e^{(\tau-\frac{\epsilon}{2})\phi}.  $$
Thus, from (\ref{similar}), we obtain that
\begin{align}
\|r^{\frac{\epsilon}{2}}e^{(\tau-2)\phi} \triangle f\|&\geq C\|r^2  e^{(\tau-\frac{\epsilon}{2})\phi}\triangle f\|\nonumber\\
 &\geq C\tau^{\frac{3}{2}} \|r^{\frac{\epsilon}{2}} e^{(\tau-\frac{\epsilon}{2})\phi} f\| \nonumber \\
&\geq C\tau^{\frac{3}{2}} \|r^\epsilon e^{\tau\phi} f\|,
\label{hih2}
\end{align}
where we have used the estimates
$$ e^{-\frac{\epsilon}{2}\phi}=r^{\frac{\epsilon}{2}} e^{-r^\epsilon}\geq r^{\frac{\epsilon}{2}} e^{-1}.   $$
Combining the inequalities (\ref{hih1}) and (\ref{hih2}), we get that
\begin{equation}
\| r^4 e^{\tau \phi} \triangle^2 f \|\geq C\tau^3 \|r^\epsilon e^{\tau\phi} f\|.
\label{nana}
\end{equation}
Carrying out the similar argument as the proof of (\ref{nana}), we can show that
\begin{equation}
\| r^4 e^{\tau \phi} \triangle^2 f \|\geq C\tau^2\delta^2 \|r^{-2} e^{\tau\phi} f\|.
\label{baba}
\end{equation}
In view of (\ref{nana}) and (\ref{baba}), we arrive at
\begin{align}
 C\| r^4 e^{\tau \phi} \triangle^2 f\|\geq \tau^3\|r^{\epsilon}e^{\tau\phi} f\|+\tau^2 \delta^2 \|r^{-2} e^{\tau\phi} f\|.
 \label{sasa}
 \end{align}
By triangle inequality and the last inequality, we deduce that
\begin{align}
C\| r^4 e^{\tau \phi} \triangle^2 f -W(x) f\|& \geq C\| r^4 e^{\tau \phi} \triangle^2 f \| -
\| r^4 e^{\tau \phi} W(x) f \| \nonumber \\
&\geq \tau^3\|r^{\epsilon}e^{\tau\phi} f\|+\tau^2 \delta^2 \|r^{-2} e^{\tau\phi} f\|\nonumber \\
&-\|W\|_{L^\infty}\| r^4 e^{\tau \phi}f \|.
\end{align}
If $\tau$ is chosen to be so large that $\tau^3\geq C(1+\|W\|_{L^\infty})$, the estimates (\ref{Carle2}) can be derived.
\end{proof}

Based on the quantitative Carleman estimates, we first deduce a quantitative three balls theorem. Let
$U=(u, v)^\intercal$, where $(u, v)$ are the solutions of the second order elliptic systems (\ref{system}). We apply such estimates to $\psi U$, where $\psi$ is an appropriate smooth cut-off function, and then select an appropriate choice of the parameter $\tau$. It is kind of a standard way to obtain the three-ball results, see e.g. \cite{AMRV}, \cite{Bak},  \cite{DF}, \cite{EV}, \cite{K}. The argument is also quite similar to the proof of Theorem \ref{th1} and the proof of Lemma \ref{baa} in the Section 5. We skip the details. The quantitative three balls theorem is stated as follows.

\begin{lemma}
There exist positive constants $\bar R$, $C$ and $0<\alpha_1<1$ which depend only on $\mathcal{M}$ such that, for any $R<\bar R$ and any $x_0\in \mathcal{M}$, the solutions $u$ of (\ref{bi-Laplace}) satisfy
\begin{equation}
\|(u, \triangle u)\|_{\mathbb B_{2R}(x_0)}\leq e^{CM^{\frac{2}{3}}} \|(u, \triangle u)\|^{\alpha_1}_{\mathbb B_R(x_0)} \|(u, \triangle u)\|^{1-\alpha_1}_{\mathbb B_{3R}(x_0)}.
\label{three-balls}
\end{equation}
\end{lemma}

We use the three balls theorem to obtain the uniform doubling inequality. Such type of argument has been performed in e.g. \cite{DF}, \cite{Bak}. We apply the arguments to elliptic systems in (\ref{system}). We establish a lower bound of $L^2$ norm of $U$ in a small ball using the overlapping of three balls argument.
\begin{lemma}
Let $u$ be the solution of bi-Laplace equations (\ref{bi-Laplace}). For any $R>0$, there exists $C_R>0$ such that
\begin{equation}\|(u, \triangle u)\|_{\mathbb B_R(x_0)}\geq e^{-C_RM^{\frac{2}{3}}} \|(u, \triangle u)\|_{L^2(\mathcal{M})} \label{lower} \end{equation}
for any $x_0\in \mathcal{M}$.
\end{lemma}
\begin{proof}
Assume that $R<\frac{R_0}{10}$. Without loss of generality, we assume that $$\|U\|_{L^2(\mathcal{M})}=\|(u, \triangle u)^\intercal\|_{L^2(\mathcal{M})} = 1.$$ We denote $y_0$ to be the point
where $$\|U\|_{\mathbb B_{2R}(y_0)}=\sup_{x\in \mathcal{M}}\|U\|_{\mathbb B_{2R}(x)}.$$ Since finite numbers of $\mathbb B_{2R}(x)$ will cover the
compact manifold $\mathcal{M}$, then $\|U\|_{\mathbb B_{2R}(y_0)}\geq D_R$, where $D_R$ depends on $\mathcal{M}$ and $R$. At any point $x\in \mathcal{M}$, the three balls theorem in the last lemma implies that
\begin{equation}\|U\|_{\mathbb B_R(x)}\geq e^{-CM^{\frac{2}{3}}} \|U\|^{\frac{1}{\alpha_1}}_{\mathbb B_{2R}(x)}.
\label{ineq}
\end{equation}
Let $l$ be the geodesic that connects $x_0$ and $y_0$. We select a sequence of $x_0, x_1, \cdots, x_m=y_0$ such that $x_i\in l$ and
$\mathbb B_R(x_{i+1})\subset \mathbb B_{2R}(x_{i})$ for $i=0, \cdots, m-1$. The number $m$ depends on the manifold $\mathcal{M}$ and $R$.
Applying the  inequality (\ref{ineq}) at $x_i$, it follows that
\begin{equation}\|U\|_{\mathbb B_R(x_i)}\geq e^{-CM^{\frac{2}{3}}} \|U\|^{\frac{1}{\alpha_1}}_{\mathbb B_{R}(x_{i+1})}
\label{mmm}
\end{equation}
for $i=0, \cdots, m-1$. Iterating the estimates (\ref{mmm}) $m$ times, it gets to the point $y_0$. Then
\begin{align*}\|U\|_{\mathbb B_R(x_0)}&\geq e^{-C_RM^{\frac{2}{3}}} \|U\|^{\frac{1}{\alpha^m_1}}_{\mathbb B_{2R}(x_{m})}&\nonumber \\&\geq e^{-C_RM^{\frac{2}{3}}} D^{\frac{1}{\alpha^m_1}}_R,   \end{align*}
which implies that
$$\|U\|_{\mathbb B_R(x_0)}\geq e^{-C_RM^{\frac{2}{3}}} \|U\|_{L^2(\mathcal{M})} .$$
Thus, the lemma is shown.
\end{proof}
Recall that $A_{R, 2R}$ is an annulus. Since the manifold $\mathcal{M}$ is complete and compact, there exists some $x_1\in A_{R, 2R}$ such that
$\mathbb B_{\frac{R}{2}}(x_1)\subset A_{R, 2R}$. From the last lemma, one has
\begin{align}
\|U\|_{R, 2R}&\geq \|U\|_{\mathbb B_{\frac{R}{2}}(x_1)} \nonumber \\
&\geq e^{-C_RM^{\frac{2}{3}}} \|U\|_{L^2(\mathcal{M})}.
\label{annu}
\end{align}

Next we proceed to show the doubling inequality.

\begin{proof}[Proof of Theorem \ref{th1}]
Let $R=\frac{\bar R}{8}$, where $\bar R$ is the fixed constant in the three balls inequality (\ref{three-balls}). Let $0<\delta<\frac{\bar R}{24}$, which can be chosen to be arbitrary small. Define a smooth cut-off function $0<\psi<1$ as follows.
\begin{itemize}
\item $\psi(r)=0$ \ \ \mbox{if} \ $r(x)<\delta$ \ \mbox{or} \  $r(x)>2R$, \medskip
\item $\psi(r)=1$ \ \ \mbox{if} \ $\frac{3\delta}{2}<r(x)<R$, \medskip
\item $|\nabla^\alpha \psi|\leq \frac{C}{R^\alpha}$ \ \ \mbox{if} $\delta<r(x)<\frac{3\delta}{2}$, \medskip
\item $|\nabla^\alpha \psi|\leq C$ \ \  \mbox{if} \ $R<r(x)<2R$,
\end{itemize}
where $\alpha=(\alpha_1, \cdots, \alpha_n)$ is a multi-index. Applying the Carleman estimates (\ref{Carle}) with $F=\psi U$ and using the elliptic systems (\ref{system}), it follows that
\begin{align*}
 \| r^{\frac{\epsilon}{2}} e^{\tau \phi}\psi U\|+ \tau\delta \| r^{-1} e^{\tau \phi}\psi U\|
\leq C\| r^2 e^{\tau \phi}(\triangle \psi U+2\nabla\psi\cdot \nabla U)\|.
\end{align*}

The properties of $\psi$ imply that
\begin{align*}
 \| r^{\frac{\epsilon}{2}} e^{\tau \phi} U\|_{\frac{R}{2}, \frac{2R}{3}}+  \|  e^{\tau \phi} U\|_{\frac{3\delta}{2}, 4\delta}
 &\leq C (\| e^{\tau \phi} U\|_{\delta, \frac{3\delta}{2}}+\| e^{\tau \phi} U\|_{R, 2R} ) \\
&+ C(\delta \| e^{\tau \phi} \nabla U\|_{\delta, \frac{3\delta}{2}}+R\| e^{\tau \phi} \nabla U\|_{R, 2R}).
\end{align*}
The radial and decreasing property of $\phi$ yields that
\begin{align*}
e^{\tau \phi(\frac{2R}{3})}\|  U\|_{\frac{R}{2}, \frac{2R}{3}}+ e^{\tau \phi({4\delta})}  \|  U\|_{\frac{3\delta}{2}, 4\delta}
&\leq C (e^{\tau \phi(\delta) }\|  U\|_{\delta, \frac{3\delta}{2}}+e^{\tau \phi(R) }\| e^{\tau \phi} U\|_{R, 2R} ) \\
&+ C(\delta e^{\tau \phi(\delta)} \| \nabla U\|_{\delta, \frac{3\delta}{2}}+Re^{\phi(R)}\| e^{\tau \phi} \nabla U\|_{R, 2R}).
\end{align*}
It is known that the Caccioppoli type inequality
\begin{equation}
\|\nabla U\|_{(1-a)r}\leq \frac{CM^{1/2}}{r}\| U\|_{r}
\label{Cacc}
\end{equation}
holds for the solution of elliptic systems (\ref{system}) with some $0<a<1$.
Using the Caccioppoli type inequality (\ref{Cacc}), we have
\begin{align}
e^{\tau \phi(\frac{2R}{3})}\|  U\|_{\frac{R}{2}, \frac{2R}{3}}+ e^{\tau \phi({4\delta})}  \|  U\|_{\frac{3\delta}{2}, 4\delta}
&\leq C M^{\frac{1}{2}} (e^{\tau \phi(\delta) }\|  U\|_{2\delta}+e^{\phi(R)}\| e^{\tau \phi}  U\|_{3R}).
\end{align}
Adding $ e^{\tau \phi({4\delta})} \|U\|_{\frac{3\delta}{2}}$ to both sides of last inequality and considering that $\phi(\delta)>\phi(4\delta)$, we obtain that
\begin{align}
e^{\tau \phi(\frac{2R}{3})}\|  U\|_{\frac{R}{2}, \frac{2R}{3}}+ e^{\tau \phi({4\delta})}  \|  U\|_{4\delta}
&\leq C M^{\frac{1}{2}} (e^{\tau \phi(\delta) }\|  U\|_{2\delta}+e^{\phi(R)}\| e^{\tau \phi}  U\|_{3R}).
\label{droppp}
\end{align}
We choose $\tau$ such that
$$CM^{\frac{1}{2}}e^{\tau\phi(R)}\|U\|_{3R}\leq \frac{1}{2}e^{\tau\phi(\frac{2R}{3})} \|U\|_{\frac{R}{2}, \frac{2R}{3}}. $$
That is,
$$ \tau\geq \frac{1}{\phi(\frac{2R}{3})-\phi(R)}\ln \frac{ 2CM^{\frac{1}{2}} \|U\|_{3R}}{ \|U\|_{\frac{R}{2}, \frac{3R}{2}}}.           $$
Then
\begin{align}
e^{\tau \phi(\frac{2R}{3})}\|  U\|_{\frac{R}{2}, \frac{2R}{3}}+ e^{\tau \phi({4\delta})}  \|  U\|_{4\delta}
&\leq C M^{\frac{1}{2}} e^{\tau \phi(\delta) }\|  U\|_{2\delta}.
\label{dropp}
\end{align}
To apply the Carleman estimates (\ref{Carle}), it is required that $\tau\geq CM^{\frac{2}{3}}$. We select
$$\tau=CM^{\frac{2}{3}}+ \frac{1}{\phi(\frac{2R}{3})-\phi(R)}\ln \frac{ 2CM^{\frac{1}{2}} \|U\|_{3R}}{ \|U\|_{\frac{R}{2}, \frac{3R}{2}}}. $$
Dropping the first term in (\ref{dropp}), we get that
\begin{align}
\|U\|_{4\delta}&\leq CM^{\frac{1}{2}} \exp\{ \big(CM^{\frac{2}{3}}+ \frac{1}{\phi(\frac{2R}{3})-\phi(R)}\ln \frac{ 2CM^{\frac{1}{2}} \|U\|_{3R}}{ \|U\|_{\frac{R}{2}, \frac{3R}{2}}}\big)\big(\phi(\delta)-\phi(4\delta)\big)  \}\|U\|_{2\delta} \nonumber \\
&\leq e^{C M^{\frac{2}{3}}}   (\frac{\|U\|_{3R}}{ \|U\|_{\frac{R}{2}, \frac{3R}{2}}})^C \|U\|_{2\delta},
\label{mada}
\end{align}
where we have used the fact that
$$ \beta_1^{-1}<\phi(\frac{2R}{3})-\phi(R)<\beta_1,  $$
$$  \beta_2^{-1}< \phi(\delta)-\phi(4\delta)<\beta_2 $$
for some positive constant $\beta_1$ and $\beta_2$ that do not depend on $R$ or $\delta$.
With aid of (\ref{annu}), it is known that
$$\frac{\|U\|_{3R}}{ \|U\|_{\frac{R}{2}, \frac{3R}{2}}}\leq e^{CM^\frac{2}{3}}. $$
Therefore, it follows from (\ref{mada}) that
$$ \|U\|_{4\delta}\leq  e^{CM^\frac{2}{3}} \|U\|_{2\delta}. $$
Choosing $\delta=\frac{r}{2}$, we obtain the doubling estimates
\begin{equation}
\|U\|_{2r}\leq  e^{CM^\frac{2}{3}} \|U\|_{r}
\end{equation}
for $r\leq \frac{\bar R}{12}$. If $r\geq \frac{\bar R}{12}$, from $(\ref{lower})$,
\begin{align*}
\|U\|_{r}&\geq \|U\|_{\frac{\bar R}{12}} \\
&\geq  e^{-C_{\bar R}M^{\frac{2}{3}}}\|U\|_{\mathcal{M}} \\
&\geq e^{-C_{\bar R}M^{\frac{2}{3}}}\|U\|_{2r}.
\end{align*}
Hence,  the doubling estimates
\begin{equation}
\|U\|_{2r}\leq  e^{CM^\frac{2}{3}} \|U\|_{r}
\end{equation}
 are achieved for any $r>0$, where $C$ only depends on the manifold $\mathcal{M}$. Since $x_0$ is any arbitrary point in $\mathcal{M}$, we have shown the uniform doubling inequality. Note that $U=(u, \triangle u)^\intercal $. The proof of Theorem \ref{th1} is arrived.
\end{proof}

\section{Refined doubling inequality for bi-Laplace equations}

This section is devoted to obtaining a refined doubling inequality for the solutions of bi-Laplace equations (\ref{bi-Laplace}).
 We apply Carleman estimates in Lemma \ref{lemmm} to show  the three balls theorem for the solution $u$ of the bi-Laplace equations (\ref{bi-Laplace}).

\begin{lemma}
There exist positive constants $\bar R$, $C$ and $0<\alpha_1<1$ which depends only on $\mathcal{M}$ such that, for any $R<\bar R$ and any $x_0\in \mathcal{M}$, the solutions $u$ of (\ref{bi-Laplace}) satisfy
\begin{equation}
\|u\|_{\mathbb B_{2R}(x_0)}\leq e^{CM^{\frac{1}{3}}} \|u\|^{\alpha_1}_{\mathbb B_R(x_0)} \|u\|^{1-\alpha_1}_{\mathbb B_{3R}(x_0)}.
\label{balls}
\end{equation}
\label{baa}
\end{lemma}

\begin{proof}
We introduce a cut-off function $\psi(r)\in C^\infty_0(\mathbb B_{3R})$ with $R<\frac{R_0}{3}$. Let $0<\psi(r)<1$ satisfy the following properties:
\begin{itemize}
\item $\psi(r)=0$ \ \ \mbox{if} \ $r(x)<\frac{R}{4}$ \ \mbox{or} \  $r(x)>\frac{5R}{2}$, \medskip
\item $\psi(r)=1$ \ \ \mbox{if} \ $\frac{3R}{4}<r(x)<\frac{9R}{4}$, \medskip
\item $|\nabla^\alpha \psi|\leq \frac{C}{R^{|\alpha|}}$
\end{itemize}
for $\alpha=(\alpha_1, \cdots, \alpha_n)$. Since the function $\psi u$ is support in the annulus $A_{\frac{R}{4}, \frac{5R}{2}}$, applying the Carelman estimates (\ref{Carle2}) with $f=\psi u$, we obtain that
\begin{align}
\tau^2 \| e^{\tau \phi} u\|& \leq C\| r^4 e^{\tau \phi}(\triangle^2 (\psi u)-W(x)\psi u)\| \nonumber  \\
&=C \| r^4 e^{\tau \phi}[\triangle^2, \ \psi] u\|,
\end{align}
where we have used the equation (\ref{bi-Laplace}). Note that $[\triangle^2, \ \psi]$ is a three order differential operator on $u$ involving the derivative of $\psi$.
By the properties of $\psi$, we have
\begin{align*}
\| e^{\tau \phi} u\|_{\frac{3R}{4}, \frac{9R}{4}}&\leq C (\| e^{\tau \phi} u\|_{\frac{R}{4}, \frac{3R}{4}}+\| e^{\tau \phi} u\|_{\frac{9R}{4}, \frac{5R}{2}} ) \\
&+ C(\sum_{|\alpha|=1}^{3} \| r^{|\alpha|} e^{\tau \phi} \nabla^\alpha u \|_{\frac{R}{4}, \frac{3R}{4}}+\sum_{|\alpha|=1}^{3} \| r^{|\alpha|} e^{\tau \phi} \nabla^\alpha u \|_{\frac{9R}{4}, \frac{5R}{2}}).
\end{align*}
Recall that the weight function $\phi$ is radial and decreasing. It follows that
\begin{align}
 \|e^{\tau \phi}  u\|_{\frac{3R}{4}, \frac{9R}{4}}&\leq C (e^{\tau \phi(\frac{R}{4})} \|  u\|_{\frac{R}{4}, \frac{3R}{4}}+e^{\tau \phi(\frac{9R}{4})} \| u\|_{\frac{9R}{4}, \frac{5R}{2}} ) \nonumber \\
&+ C(e^{\tau \phi(\frac{R}{4})}\sum_{|\alpha|=1}^{3}\| r^{|\alpha|} \nabla^\alpha u\|_{\frac{R}{4}, \frac{3R}{4}}+e^{\tau \phi(\frac{9R}{4})}\sum_{|\alpha|=1}^{3}\| r^{|\alpha|} \nabla^\alpha u\|_{\frac{9R}{4}, \frac{5R}{2}}).
\label{recall}
\end{align}
For the higher order elliptic equations
\begin{equation} (- \triangle)^m u+W(x) u=0,
\label{zhu}
\end{equation}
the following Caccioppoli type inequality
\begin{equation}
\sum^{2m-1}_{|\alpha|=0} \|r^{|\alpha|}\nabla^\alpha u\|_{c_3R,  c_2R} \leq C (\|W\|_{L^\infty}+1)^{2m-1} \|u\|_{c_4R,  c_1R}
\label{hihcac}
\end{equation}
has been shown in \cite{Zhu} for all positive constant $0<c_4<c_3<c_2<c_1<1$.
The estimate (\ref{hihcac}) yields that
\begin{align*}
\| r^{|\alpha|} \nabla^\alpha u\|_{\frac{R}{4}, \frac{3R}{4}}\leq CM^3 \| u\|_{R}
\end{align*}
and
\begin{align*}
\| r^{|\alpha|} \nabla^\alpha u\|_{\frac{9R}{4}, \frac{5R}{2}}\leq  C M^3 \| u\|_{3R}
\end{align*}
for all $1\leq |\alpha|\leq 3$. Therefore, from (\ref{recall}), we get that
\begin{align}
\|u\|_{\frac{3R}{4}, 2R}\leq CM^{3}_1\big( e^{\tau(\phi(\frac{R}{4})-\phi(2R))} \|u\|_R+  e^{\tau(\phi(\frac{9R}{4})-\phi(2R))} \|u\|_{3R}\big).
\label{inequal}
\end{align}
We choose parameters
$$ \beta^1_R=\phi(\frac{R}{4})-\phi(2R),$$
$$ \beta^2_R=\phi(2R)-\phi(\frac{9R}{4}).$$
From the definition of $\phi$, we know that
$$ 0<\beta^{-1}_1<\beta^1_R<\beta_1 \quad \mbox{and} \quad 0<\beta_2<\beta^2_R<\beta^{-1}_2,$$
where $\beta_1$ and $\beta_2$ do not depend on $R$.
Adding $\|u\|_{\frac{3R}{4}}$ to both sides of the inequality (\ref{inequal}) gives that
\begin{equation}
\|u\|_{2R}\leq CM^{3}\big( e^{\tau\beta_1}\|u\|_R+  e^{-\tau\beta_2}\|u\|_{3R}     \big).
\end{equation}
To incorporate the second term in the right hand side of the last inequality into the left hand side,
 we choose $\tau$ such that
$$CM^{3}e^{-\tau\beta_2}\|u\|_{3R}\leq \frac{1}{2}\|u\|_{2R},   $$
which is true if
$$\tau\geq \frac{1}{\beta_2} \ln \frac{2CM^{3}\|u\|_{3R}}{\|u\|_{2R} }.   $$
Thus, we obtain that
\begin{equation}
\|u\|_{2R}\leq CM^3 e^{\tau\beta_1}\|u\|_R.
\label{substitute}
\end{equation}
Since $\tau>CM^{3}$ is needed to apply the Carleman estimates (\ref{Carle2}), we choose
$$ \tau=CM^{\frac{1}{3}}+\frac{1}{\beta_2} \ln \frac{2CM^3\|u\|_{3R}}{\|u\|_{2R} }. $$
Substituting such $\tau$ in (\ref{substitute}) gives that
\begin{align}
\|u\|_{2R}^{\frac{\beta_2+\beta_1}{\beta_2}} \leq e^{ CM^{\frac{1}{3}}}\|u\|_{3R}^{\frac{\beta_1}{\beta_2}} \|u\|_R.
\end{align}
Raising exponent $\frac{\beta_2}{\beta_2+\beta_1}$ to both sides of the last inequality yields that
\begin{align}
\|u\|_{2R} \leq e^{ CM^{\frac{1}{3}}}\|u\|_{3R}^{\frac{\beta_1}{\beta_1+\beta_2}} \|u\|_R^{\frac{\beta_2}{\beta_1+\beta_2}}.
\end{align}
Set $\alpha_1={\frac{\beta_2}{\beta_1+\beta_2}}$. We arrive at the three balls inequality in the lemma.
\end{proof}

Following the strategy in the proof of (\ref{lower}) and (\ref{annu}) by using the three balls theorem (\ref{balls}), we can show the following results. For any $R>0$, there exists $C_R$ such that
\begin{equation} \|u\|_{\mathbb B_R(x_0)}\geq e^{-C_RM^{\frac{1}{3}}} \|u\|_{L^2(\mathcal{M})} \end{equation}
for any $x_0\in \mathcal{M}$. Furthermore, it holds that
\begin{align}
\|u\|_{R, 2R}\geq e^{-C_RM^{\frac{1}{3}}} \|u\|_{L^2(\mathcal{M})}.
\label{annu11}
\end{align}

Next we proceed to show the doubling inequality for the solutions of bi-Laplace equations (\ref{bi-Laplace}). The argument is somewhat parallel to the proof of the double inequality for elliptic systems. We show the details of the argument as follows.

\begin{proof}[Proof of Theorem \ref{th3}]
Let us fix $R=\frac{\bar R}{8}$, where $\bar R$ is the one in the three balls inequality (\ref{balls}). Let $0<\delta<\frac{R}{24}$ be arbitrary small. A smooth cut-off function $0<\psi<1$ is introduced as follows,
\begin{itemize}
\item $\psi(r)=0$ \ \ \mbox{if} \ $r(x)<\delta$ \ \mbox{or} \  $r(x)>2R$, \medskip
\item $\psi(r)=1$ \ \ \mbox{if} \ $\frac{3\delta}{2}<r(x)<R$, \medskip
\item $|\nabla^\alpha \psi|\leq \frac{C}{R^\alpha}$ \ \ \mbox{if} $\delta<r(x)<\frac{3\delta}{2}$,\medskip
\item $|\nabla^\alpha \psi|\leq C$ \ \  \mbox{if} \ $R<r(x)<2R$.
\end{itemize}
We use the Carleman estimates (\ref{Carle2}) again. Replacing $f$ by $\psi u$ and substituting it into (\ref{Carle2}) gives that
\begin{align*}
 \| r^{\epsilon} e^{\tau \phi}\psi u\|+ \tau^2\delta^2 \| r^{-2} e^{\tau \phi}\psi u\|
\leq C\| r^4 e^{\tau \phi}[ \triangle^2, \ \psi] u\|,
\end{align*}
where $[\triangle^2, \ \psi]$ is a three order differential operator on $u$ involving the derivative of $\psi$ and $\tau \geq 1$. The properties of $\psi$ imply that
\begin{align*}
 \| r^{\epsilon} e^{\tau \phi} u\|_{\frac{R}{2}, \frac{2R}{3}}+  \|  e^{\tau \phi} u\|_{\frac{3\delta}{2}, 4\delta}
 &\leq C (\| e^{\tau \phi} u\|_{\delta, \frac{3\delta}{2}}+\| e^{\tau \phi} u\|_{R, 2R} ) \\
&+ C(\sum^{3}_{|\alpha|=1}\| r^{|\alpha|} e^{\tau \phi} \nabla^\alpha u\|_{\delta, \frac{3\delta}{2}}+\sum^{3}_{|\alpha|=1}\|  r^{|\alpha|} e^{\tau \phi} \nabla^\alpha u\|_{R, 2R}).
\end{align*}
Taking the exponential function $e^{\tau\phi}$ out by using the fact that $\phi$ is radial and decreasing, we obtain that
\begin{align*}
e^{\tau \phi(\frac{2R}{3})}\|  u\|_{\frac{R}{2}, \frac{2R}{3}}+ e^{\tau \phi({4\delta})}  \|  u\|_{\frac{3\delta}{2}, 4\delta}
&\leq C (e^{\tau \phi(\delta) }\|  u\|_{\delta, \frac{3\delta}{2}}+e^{\tau \phi(R) }\| e^{\tau \phi} u\|_{R, 2R} ) \\
&+ C(e^{\tau \phi(\delta)}\sum^{3}_{|\alpha|=1} \| r^{|\alpha|} e^{\tau \phi} \nabla^\alpha u\|_{\delta,
\frac{3\delta}{2}}+e^{\phi(R)}\sum^{3}_{|\alpha|=1}\|r^{|\alpha|} e^{\tau \phi} \nabla^\alpha u\|_{R, 2R}).
\end{align*}
The use of  Caccioppoli type inequality (\ref{hihcac})
further implies that
\begin{align}
e^{\tau \phi(\frac{2R}{3})}\|  u\|_{\frac{R}{2}, \frac{2R}{3}}+ e^{\tau \phi({4\delta})}  \|  u\|_{\frac{3\delta}{2}, 4\delta}
&\leq C M^3 (e^{\tau \phi(\delta) }\|  u\|_{2\delta}+e^{\phi(R)}\| e^{\tau \phi} u\|_{3R}).
\end{align}
Adding $ e^{\tau \phi({4\delta})} \|u\|_{\frac{3\delta}{2}}$ to both side of last inequality, it follows that
\begin{align}
e^{\tau \phi(\frac{2R}{3})}\|  u\|_{\frac{R}{2}, \frac{2R}{3}}+ e^{\tau \phi({4\delta})}  \|  u\|_{4\delta}
&\leq C M^3 (e^{\tau \phi(\delta) }\|  u\|_{2\delta}+e^{\phi(R)}\| e^{\tau \phi} u\|_{3R}).
\label{drop}
\end{align}
We want to get rid of the second term in the right hand side of the last inequality.
We choose $\tau$ such that
$$CM^3 e^{\tau\phi(R)}\|u\|_{3R}\leq \frac{1}{2}e^{\tau\phi(\frac{2R}{3})} \|u\|_{\frac{R}{2}, \frac{2R}{3}}.  $$
That is, at least
$$ \tau\geq \frac{1}{\phi(\frac{2R}{3})-\phi(R)}\ln \frac{ 2CM^3 \|u\|_{3R}}{ \|u\|_{\frac{R}{2}, \frac{3R}{2}}}.    $$
Then we arrive at
\begin{equation}
e^{\tau \phi(\frac{2R}{3})}\|  u\|_{\frac{R}{2}, \frac{2R}{3}}+ e^{\tau \phi({4\delta})}  \|  u\|_{4\delta}
\leq C M^{3} e^{\tau \phi(\delta) }\|  U\|_{2\delta}.
\label{droppd}
\end{equation}

To apply the Carleman estimates (\ref{Carle2}), the assumption for $\tau$ is that  $\tau\geq CM^\frac{1}{3}$. Therefore, we select
$$\tau=CM^\frac{1}{3}+ \frac{1}{\phi(\frac{2R}{3})-\phi(R)}\ln \frac{ 2CM^3 \|u\|_{3R}}{ \|u\|_{\frac{R}{2}, \frac{3R}{2}}}. $$
Furthermore, dropping the first term in (\ref{droppd}), we get that
\begin{align}
\|u\|_{4\delta}&\leq CM^3 \exp\{ \big(CM^\frac{1}{3}+ \frac{1}{\phi(\frac{2R}{3})-\phi(R)}\ln \frac{ 2CM^3 \|u\|_{3R}}{ \|u\|_{\frac{R}{2}, \frac{3R}{2}}}\big)\big(\phi(\delta)-\phi(4\delta)\big)  \}\|u\|_{2\delta} \nonumber \\
&\leq e^{C M^\frac{1}{3}}   (\frac{\|u\|_{3R}}{ \|u\|_{\frac{R}{2}, \frac{3R}{2}}})^C \|u\|_{2\delta}.
\label{tata}
\end{align}
It follows from (\ref{annu11}) that
$$\frac{\|u\|_{3R}}{ \|u\|_{\frac{R}{2}, \frac{3R}{2}}}\leq e^{CM^\frac{1}{3}}. $$
Combining the last inequality with (\ref{tata}) yields that
$$ \|u\|_{4\delta}\leq  e^{CM^\frac{1}{3}} \|u\|_{2\delta}. $$
Let $\delta=\frac{r}{2}$. The doubling inequality
\begin{equation}
\|u\|_{2r}\leq  e^{CM^\frac{1}{3}} \|u\|_{r}
\end{equation}
is deduced for $r\leq \frac{R_0}{12}$. If $r\geq \frac{R_0}{12}$, using (\ref{annu11}) as the arguments analogous to the elliptic systems, we can derive that
\begin{equation}
\|u\|_{{2r}}\leq  e^{CM^\frac{1}{3}} \|u\|_{{r}}
\end{equation}
for any $r>0$ and $x_0\in \mathcal{M}$, where $C$ only depends on the manifold $\mathcal{M}$. Therefore, the theorem is completed.
\end{proof}

At last, we give the proof of the corollary \ref{cor1} based on the doubling inequality in Theorem \ref{th3}.
\begin{proof}[Proof of Corollary \ref{cor1}]
The $L^\infty$ norm estimate for higher order elliptic equations (\ref{hihcac}) was shown in \cite{Zhu},
\begin{equation}
\|u\|_{L^\infty(\mathbb B_r)} \leq C
(\|W\|_{L^\infty}+1
)^{\frac{n}{2}} r^{-\frac{n}{2}}\|u\|_{L^2 (\mathbb B_{2r})}.
\end{equation}
Thus, we can see that Theorem \ref{th3} implies the doubling inequality with $L^\infty$ norm
\begin{equation}
\|u\|_{L^\infty(\mathbb B_{2r}(x))}\leq e^{ CM^{\frac{1}{3}}}\|u\|_{L^\infty(\mathbb B_{r}(x))}
\label{comeon}
\end{equation}
for any $x\in \mathcal{M}$ and $0<r<r_0$, where $r_0$ depends only on $\mathcal{M}$. We may assume that $\|u\|_{L^\infty(\mathcal{M})}=1$. So there exists some point $x_0$ such that $\|u\|_{L^\infty(\mathcal{M})}=|u(y_0)|=1$. For any point $x_0\in \mathcal{M}$, there exists a geodesic $l$ connecting $x_0$ and $y_0$. We choose a sequence of point $x_0, x_1, \cdots, x_m=y_0$ such that $x_i\in l$ and $\mathbb B_r(x_{i+1})\subset \mathbb B_{2r}(x_{i})$ for $i=0, \cdots, {m-1}$.  It is true that the number $$m\leq C\log_2\frac{\diam{\mathcal{M}}}{r}.$$ Applying the $L^\infty$ norm of the doubling inequality with iteration and using the fact that $$\|u\|_{L^\infty(\mathbb B_r(x_{i+1}))}\leq \|u\|_{L^\infty(\mathbb B_{2r}(x_{i}))},$$ we obtain that
\begin{align}
\|u\|_{L^\infty(\mathbb B_r(x_0))}&\geq e^{-CM^{\frac{1}{3}}\log_2\frac{\diam{\mathcal{M}}}{r}}\|u\|_{L^\infty(\mathbb B_r(y_0))} \nonumber \\
&\geq C r^{CM^{\frac{1}{3}}},
\end{align}
where $C$ depends on the manifold $\mathcal{M}$. This implies that the vanishing order of solution is less than $CM^{\frac{1}{3}}$. Since $x_0$ is an arbitrary point, we get such vanishing rate of solutions for every point on the manifold $\mathcal{M}$. Therefore, we complete the proof of the corollary.
\end{proof}

\section{Implicit upper bound of nodal sets}
In this section, we obtain an upper bound for the nodal sets of bi-Laplace equation (\ref{bi-Laplace}). Such type bound have been obtained for the measure of singular sets for semi-linear elliptic equations and higher order elliptic equations by Han, Hardt and Lin \cite{HHL}, \cite{HHL1}. The method is based on a compactness argument and an iteration procedure.  The iteration argument was first developed by Hart and Simon \cite{HS}. We adapt such compactness argument to obtain the measure of nodal sets for (\ref{bi-Laplace}). For higher order elliptic equations, it seems hard to get Hart and Simon's exponential upper bound result for nodal sets, even if the explicit vanishing order is achieved, since nodal sets comparison lemma in \cite{HS} is not known for higher order derivatives.

 The method applies to higher order elliptic equations without variational structure. Hence
we consider general fourth order homogeneous elliptic equations in $\mathbb B_1(0)\subset \mathbb R^n$ given by
\begin{equation}
Lu=\sum^4_{|\nu|=0} a_{\nu}(x) D^\nu u=0,
\label{impli}
\end{equation}
where $a_{\nu}(x)$ is a smooth function for $|\nu|\geq 1$,  $a_0(x)\in L^\infty$ and
\begin{align}
\sum_{|\nu|=4} a_{\nu}(x) \xi^\nu\geq \Lambda \ \mbox{for any} \ \xi\in S^{n-1} \ \mbox{and} \ x\in \mathbb B_1(0)
\label{ellip}
\end{align}
for some positive constant $\Lambda$.
 It is easy to observe that the equation (\ref{bi-Laplace}) we are considering is a particular case of the equations $Lu=0$ in (\ref{impli}). We say the operator $L\in \mathcal{L}(\Lambda, K)$ if $L$ is given by
(\ref{impli}) satisfying (\ref{ellip}) and
\begin{align}
\sum^4_{|\nu|=1} \| a_{\nu}\|_{C^\infty(\mathbb B_1)}+ \| a_0\|_{L^\infty(\mathbb B_1)}\leq K
\end{align}
for some positive constant $K$. By the standard elliptic estimates, we have
\begin{equation}
\|u\|_{C^{3, \alpha}(\mathbb B_{1-r})}\leq C\|u\|_{L^2(\mathbb B_1)}
\end{equation}
for some $0<\alpha<1$, where $C$ depends on $K$, $r$ and $n$.

We consider the geometric structure of nodal sets $\mathcal{N}(u)=\{\mathbb B_{1/2}| u(x)=0\}$. Let $\mathcal{O}(p)$ denote the vanishing order of $u$ at $p$. Then $\mathcal{N}(u)=\{p\in \mathbb B_1 : \mathcal{O}(p)\geq 1\}$. For each integer $d\geq 1$, define the $d$th level set as
\begin{align}
\mathcal{L}_d(u)=\{p\in \mathbb B_1: \mathcal{O}(p)=d\}.
\end{align}
Thus, we can write
\begin{equation}
\mathcal{N}(u)=\cup_{d\geq 1} \mathcal{L}_d(u).
\end{equation}
The following lemma shows that the Hausdorff dimension of nodal sets and the property of leading polynomials at the $n-1$ dimensional nodal sets. The lemma is directly from the Theorem 5.1 in \cite{Han}. We present most of the proof for the complete of presentation.
\begin{lemma}
If the solution $u$ satisfies (\ref{impli}) and does not vanish of infinite order, then $\mathcal{N}(u)$ is countably $(n-1)$-rectifiable. Furthermore, for $H^{n-1}$ almost all points in $\mathcal{N}(u)$, the leading polynomials of the solutions are functions of one variable after an appropriate rotation.
\label{haudim}
\end{lemma}
\begin{proof}
 For each $y\in \mathbb B_{1/2}(0)\cap \mathcal{L}_d(u)$, set
 $$ u_{y, r}(x)=\frac{u(y+rx)}{(  \int_{\partial \mathbb B_r(y)} u^2)^{1/2}},\quad \quad x\in\mathbb B_2(0)$$
 for $r\in (0, \ \frac{1-|y|}{2})$. By Theorem 3.3 in \cite{Han},
 $$u_{y, r}\to P \ \mbox{in} \ L^2(\mathbb B_2(0)) \ \mbox{as}\  r \to 0. $$
 The  homogeneous polynomial $P=P_y$ satisfies
\begin{align}
\sum_{|\nu|=4} a_{\nu}(0) D^\nu P=0.
\label{eqnpoly}
\end{align}
$P$ is called
 the leading polynomial of $u$ at $y$.
Since $P$ is $d$ degree non-zero homogeneous polynomial, we introduce
\begin{align}
\mathcal{L}_d(P)=\{x| D^\nu P(x)=0 \ \mbox{for any} \ |\nu|\leq d-1\}.
\end{align}
Clearly, $\mathcal{L}_d(P)$ is not an empty set, since $0\in \mathcal{L}_d(P)$. We claim that $\mathcal{L}_d(P)$ is a linear subspace and
\begin{equation}
P(x)=P(x+z)
\label{polyno}
\end{equation}
for any $x\in \mathbb R^n$ and $z\in \mathcal{L}_d(P)$. Since $z\in \mathcal{L}_d(P)$, then
\begin{align*}
D^\nu P(z)=0 \ \mbox{for any} \ |\nu|\leq d-1.
\end{align*}
It is assumed that
\begin{align*}
P(x)= \sum_{|\alpha|=d} a_\alpha x^\alpha.
\end{align*}
Hence it is true that
\begin{align*}
P(x)= \sum_{|\alpha|=d} a_\alpha (x-z)^\alpha,
\end{align*}
which implies the identity (\ref{polyno}). Furthermore, it is easy to see that $\mathcal{L}_d(P)$ is a linear space. From the formula (\ref{polyno}), we also know that the polynomial $P$ is a function of dimension $n-$dim$\mathcal{L}_d(P)$ variables. Observe that dim$\mathcal{L}_d(P)\leq n-1$ and that dim$\mathcal{L}_d(P)\leq n-2$ for $d\geq 4$.  If dim$\mathcal{L}_d(P)=n-1$, then $P$ is a $d$-degree monomial of one variable satisfying (\ref{eqnpoly}). Then $d<4$.

We define
\begin{align}
\mathcal{L}_d^j(u)=\{ y\in \mathcal{L}_d(u); dim \mathcal{L}_d(P_y)=j\}
\end{align}
for $i=0, 1, \cdots, n-1$. Following the arguments in \cite{Han}, we can show that $\mathcal{L}_d^j$ is on a countable union of $j$-dimensional $C^1$ graphs. Next, we show that $\mathcal{L}_d^{n-1}(u)$ is on a countable union of $(n-1)$ dimensional $C^{1, \frac{\alpha}{d}}$ graphs for $d=1, 2, 3$. Let $y=0\in \mathcal{L}_d^{n-1}(u)$, by  denoting $\mathbb R^n=R^1\times \mathcal{L}_d(P)$ and the argument discussed before, $P$ is a monomial of degree $d$ in $R^1$. After an appropriate rotation, there holds
\begin{align}
u(x)=c x_1^d+\psi(x) \quad \mbox{in} \ \mathbb B_{1/2}.
\end{align}
The function $\psi$ satisfies
\begin{align}
|D^i\psi(x)|\leq C|x|^{d-i+\alpha} \ \mbox{for}  \ i=0, 1, \cdots, d
\label{psi1}
\end{align}
and
\begin{align}
|D^i\psi(x)|\leq C \ \mbox{for}  \ i=d+1, \cdots, 3.
\label{psi2}
\end{align}
For $x\in \mathcal{L}_d^{n-1}(u)\cap \mathbb B_{1/2}$, since $u(x)=0$, there holds
$$|x_1|^d\leq C|x|^{d+\alpha}.$$
Hence, the local $(n-1)$ dimensional $C^1$ graph containing $\mathcal{L}_d^{n-1}(u)$ is $C^{1, \frac{\alpha}{d}}$ in a neighborhood of $0$.
Let $\mathcal{L}^j (u)=\cup_{d\geq 1} \mathcal{L}^j_d (u)$ for $j=0, 1, \cdots, n-1$. Then
\begin{align}
\mathcal{N}(u)= \cup^{n-1}_{j=0} \mathcal{L}^j (u).
\end{align}
Each $ \mathcal{L}^j (u)$ is on a countable union of $j$-dimensional $C^1$ manifolds for $j=0, \cdots, n-1$. Set
\begin{align}
\mathcal{N}_\ast (u)= \cup^{n-2}_{j=0} \mathcal{L}^j (u), \\
\mathcal{N}^\ast (u)=  \mathcal{L}^{n-1} (u).
\end{align}
Then we have the decomposition
\begin{align}
\mathcal{N}(u)=\mathcal{N}^\ast (u)\cup \mathcal{N}_\ast (u),
\label{NNN}
\end{align}
where $\mathcal{N}_\ast (u)$ is countably $(n-2)$-rectifiable and $\mathcal{N}^\ast (u)$ is on a countable union of $(n-1)$ dimensional $C^{1, \alpha}$ manifold.
Note that for $y\in \mathcal{N}^\ast (u)$, the leading polynomial $P$ of $u$ at $y$ is a homogeneous with one variable.
\end{proof}

The next proposition states that the nodal sets can be decomposed into a good part and a bad part. The good part has a measurable upper estimate and the bad part is covered by the small balls.

\begin{proposition}
There exist positive constants $C(u)$ and $\epsilon(u)$ depending on the solution $u$ and a finite collection of balls $\mathbb B_{r_i}(x_i)$ with $r_i\leq \frac{1}{10}$ and $x_i\in \mathcal{N}(u)$ such that for any  $v\in C^3$ with
\begin{align}
\|u-v\|_{C^3(\mathbb B_1)}\leq \epsilon(u),
\end{align}
there holds
\begin{align}
H^{n-1}\big(\mathcal{N}(v)\cap \mathbb B_{1/2}\backslash \bigcup B_{r_i}(x_i)\big)< C(u)
\end{align}
and
$$\sum r_i^{n-1} \leq \frac{1}{2^n},  $$
where $C(u)$ depends on $u$ and coefficients of the operator $L$.
\label{prohau}
\end{proposition}

\begin{proof}
It follows from the relation (\ref{NNN}) that the set $\mathcal{N}_\ast(u)$ has dimension less than $n-1$. Thus,
\begin{equation}
H^{n-1}(\mathcal{N}_\ast(u))=0.
\end{equation}
By the definition of Hausdorff measure, there exist at most countably many balls $\mathbb B_{r_i}(x_i)$ with $r_i\leq \frac{1}{10}$ and $x_i\in \mathcal{N}_\ast(u)$ such that
\begin{align}
\mathcal{N}_\ast(u)\subset \cup_{i} \mathbb B_{r_i}(x_i)
\end{align}
and
\begin{align}
\sum r_i^{n-1}\leq \frac{1}{2^n}.
\end{align}
We consider the set $\mathcal{N}^\ast(u)\cap \mathbb B_{3/4}$. We claim that, for any $y\in \mathcal{N}^\ast(u)\cap \mathbb B_{3/4}$, there exist positive constants $R(y, u)<\frac{1}{10}$, $r(y, u)$, $\delta=\delta(y, u)$ and $C=C(y, u)$ with $r<R$ such that
\begin{align}
H^{n-1}\big(  \mathcal{N}(v)\cap \mathbb B_{r}(y)\big)\leq C r^{n-1},
\label{sclaim}
\end{align}
if the function $v$ satisfies
\begin{align}
\| u-v\|^\ast_{C^3(\mathbb B_R(y))}\leq \delta.
\label{under}
\end{align}
Here the norm $\|\cdot\|^\ast_{C^m(\mathbb B_R)}$ is defined as
\begin{align*}
\|f\|^\ast_{C^m(\mathbb B_R)}=\sum^m_{i=0} R^i \sup_{x\in \mathbb B_R}| D^i f(x)|
\end{align*}
for any $f\in C^m(\mathbb B_R)$.

By the compactness of $\mathcal{N}(u)$, there exist $x_i\in \mathcal{N}_\ast(u)$ and $y_i \in \mathcal{N}^\ast(u)$ for $i=1, \cdots, m(u)$ and $j=1, \cdots, k(u)$ such that
\begin{align}
\mathcal{N}(u)\cap  \mathbb B_{3/4}\subset (\bigcup_{i=1}^{m(u)} \mathbb B_{r_i}(x_i)) \cap (\bigcup_{j=1}^{k(u)} \mathbb B_{s_i}(y_i))
\end{align}
with $r_i\leq \frac{1}{10}$ and $s_i\leq \frac{1}{10}$. By the compactness of $\mathcal{N}(u)$ again,
there exists a positive constant $\rho=\rho(u)$ such that
\begin{align}
\{ x\in \mathbb B_{3/4}; \dist(x, \mathcal{N}(u))< \rho\}\subset (\bigcup_{i=1}^{m(u)} \mathbb B_{r_i}(x_i)) \cap (\bigcup_{j=1}^{k(u)} \mathbb B_{s_i}(y_i)).
\end{align}
For such a $\rho$, we can find a positive constant $\eta=\eta(u)$ such that
\begin{align}
\mathcal{N}(v)\cap  \mathbb B_{1/2} \subset \{ x\in \mathbb B_{3/4}; \dist(x, \mathcal{N}(u))< \rho\}
\end{align}
if $\|u-v\|_{C^1(\mathbb B_{3/4})}<\eta$. For the convenience of the presentation, let
\begin{align*}
\mathcal{B}_u^1 =\cup^m_{i=1} \mathbb B_{r_i}(x_i), \quad \quad \mathcal{B}_u^2 =\cup^k_{i=1} \mathbb B_{s_i}(y_i).
\end{align*}
We take $\epsilon(u)<\eta(u)$ small enough. For any $v\in C^3$ in $\mathbb B_1$, if
\begin{align}
\|u-v\|_{C^3}<\epsilon(u),
\label{ceps}
\end{align}
then
\begin{align}
\|u-v\|^\ast_{C^3(\mathbb B_R(y_i))}<\delta(y_i, u)
\end{align}
for $i=1, \cdots, k=k(u)$. Thus, from the previous arguments and (\ref{sclaim}), we obtain that
\begin{align}
\mathcal{N}(v)\cap  \mathbb B_{1/2}\subset ( \mathcal{N}(v)\cap \mathcal{B}_u^1)\bigcup ( \mathcal{N}(v)\cap \mathcal{B}_u^2)
\end{align}
and
\begin{align}
H^{n-1}(\mathcal{N}(v)\cap  \mathcal{B}_u^2)\leq C\sum^{k(u)}_{j=1} s_j^{n-1}=C(u).
\label{reca1}
\end{align}
Recall that
\begin{align}
\mathcal{B}_u^1 =\cup^m_{i=1} \mathbb B_{r_i}(x_i) \ \mbox{with} \ \sum^k_{i=1} r_i^{n-1}\leq \frac{1}{2^n}.
\label{reca2}
\end{align}
Hence the proof of the theorem follows from (\ref{reca1}) and (\ref{reca2}). We are left to prove the claim (\ref{sclaim}). Thanks to the arguments in Lemma \ref{haudim}, for any $y\in \mathcal{N}^\ast(u)\cap \mathbb B_{3/4}$, there holds
\begin{equation}
u(x+y)=P(x)+\psi(x),
\end{equation}
where $P$ is a non-zero $d$-degree monomial with $1\leq d\leq 3$ and $\psi$ satisfies (\ref{psi1}) and (\ref{psi2}). Thus, we can take a positive constant $R=R(y, u)<\frac{1}{10}$ such that
\begin{align}
\|\frac{1}{R^d} \psi\|^\ast_{C^3(\mathbb B_R)}<\frac{\epsilon_\ast}{2}.
\end{align}
Choosing $\delta$ so small that (\ref{under}) implies that
\begin{align}
\|\frac{1}{R^d} (u-v)\|^\ast_{C^3(\mathbb B_R(y))}<\frac{\epsilon_\ast}{2},
\end{align}
then there holds that
\begin{align}
\|\frac{1}{R^d} (v-P(\cdot-y))\|^\ast_{C^3(\mathbb B_R(y))}<\epsilon_\ast.
\end{align}
By considering the transformation $x\to y+Rx$, we obtain that
\begin{align}
\|\frac{1}{R^d} v(y+R\cdot)-P\|_{C^3(\mathbb B_1)}<\epsilon_\ast.
\end{align}
Since $P=C x_1^d$ for $1\leq d\leq 3$, we can find an orthonormal basis $\{ e_1, \cdots, e_n\}$ in $\mathbb B_1$ such that
\begin{align}
D_{e_i}^d(P) \ \mbox{is a nonzero constant for any} \ i=1, \cdots, n.
\end{align}
Therefore, there exist positive constants $r=r(y, P)$ and $\epsilon_\ast=\epsilon_\ast(y, P)$ such that if the function $v\in C^d$ satisfies
\begin{align}
\| P-\frac{1}{R^d} v(y+ R\cdot)\|_{C^3(\mathbb B_r)}\leq \epsilon_\ast,
\end{align}
then $D_{e_i}^d v(y+R\cdot)$ is never zero in $\mathbb B_r(y)$ for any $i=1, \cdots, n$. By using one dimensional mean value theorem $d$ times, we conclude that there can not be more than $d+1$ zeros for $\frac{1}{R^d} v(y+ R\cdot)$ in any line parallel to $e_i$ for any $i=1, \cdots, n$. Let $z_i$ be the variable in the $e_i$ direction. We set $\pi_i$ as the projection
\begin{align*}
\pi_i(z_1, z_2,\cdots, z_n)=(z_1, \cdots, z_{i-1}, z_{i+1}, \cdots, z_n)\in \mathbb R^{n-1}.
\end{align*}
Denote $\frac{1}{R^d} v(y+ R\cdot)$ as $\tilde{v}$. Thus, for any $q\in \mathbb B_r^{n-1}\subset \mathbb R^{n-1}$ and $1\leq i\leq n$, we have
\begin{align*}
card( \tilde{v}^{-1}(0)\cap \pi^{-1}_i(q)\cap \mathbb B_r)\leq (d+1).
\end{align*}
From the integral geometric formula 3.2.22 in \cite{F}, we derive that
\begin{align}
H^{n-1}(\tilde{v}^{-1}(0)\cap \mathbb B_r) &\leq \sum_{1\leq i\leq n}\int_{\mathbb B_r^{n-1}} card\big( \tilde{v}^{-1}(0)\cap \pi^{-1}_i(q)\cap \mathbb B_r\big) \,d H^{n-1} \nonumber\\
& \leq C(n)(d+1) r^{n-1}.
\end{align}
See the similar arguments in \cite{Yo}. After transforming back to $\frac{1}{R^d} v(y+ R\cdot)$ in  $\mathbb B_R(y)$, we have for $r\leq R r_\ast$,
\begin{align}
H^{n-1}(v^{-1}(0)\cap \mathbb B_r(y))\leq C r^{n-1}.
\end{align}
Thus, the claim (\ref{sclaim}) follows. Therefore, the proposition is shown.

\end{proof}

We consider the translation and rescaling property of the operator $L$.
Let $L_{x_0, \rho}$ be defined by
\begin{align*}
L_{x_0, \rho}=\sum^4_{|\nu|=0} \rho^{4-|\nu|} a_\nu(x_0+\rho x) D^\nu.
\end{align*}
Observe that $L_{x_0, \rho}\in \mathcal{L}(\Lambda, K)$.

To control the vanishing order quantitatively, we introduce the quantitative doubling inequality. A function is said to be in $D_N$ if
\begin{align}
\|u\|_{L^2(\mathbb B_{2r}(x_0))} \leq 2^N \|u\|_{L^2(\mathbb B_{r}(x_0))}
\label{DNDN}
\end{align}
for $x_0\in \mathbb B_{2/3}$ and $0<2r<\dist(x_0, \ \partial\mathbb B_1)$.
We define $D_N^\ast$ as the collection of all functions $u$ in $D_N$ satisfying $Lu=0$ in $\mathbb B_1$ for some $L\in \mathcal{L}(\Lambda, K)$.
By the standard elliptic estimates, the collection
$$\{ u\in D_N^\ast; \int_{\mathbb B_{1/2}}u^2\, dx=1\}  $$
is compact under the local $L^\infty$ metric. See the lemma 4.1 in \cite{HHL1}. Next we show the upper bound estimates of nodal sets by removing a finite collection of small balls.
\begin{lemma}
There exists $C$ depending on $K$, $N$ and $\lambda$ such that for any $u\in D_N^\ast$, there exists a finite collection of balls $\{\mathbb B_{r_i}(x_i)\},$ with $r_i\leq \frac{1}{4}$ and $x_i\in \mathcal{N}(u)$ such that there hold
\begin{align}
H^{n-1}\big( \mathcal{N}(u)\cap \mathbb B_{1/2}\backslash \bigcup\mathbb B_{r_i}(x_i)\big)\leq C
\end{align}
and
$$ \sum r^{n-1}_i\leq \frac{1}{2}. $$
\label{le13}
\end{lemma}
\begin{proof}
Define $D^1_N$ to be the set
$$\{ u\in D_N^\ast; \int_{\mathbb B_{1/2}}u^2\, dx=1\}.  $$
Let $u_0$ be an arbitrary solution in $D^1_N$. For any $u\in D^1_N$, if $\|u-u_0\|_{L^\infty (\mathbb B_{7/8})}\leq \epsilon_0$, by standard elliptic estimates,
$$\|u-u_0\|_{C^{3, \alpha} (\mathbb B_{3/4})}\leq C(\epsilon_0), $$
where $C(\epsilon_0)\to 0$ as $\epsilon_0\to 0$. We can take $\epsilon_0$ small so that $C(\epsilon_0)<\epsilon(u_0)$, where $\epsilon(u_0)$ is the constant in (\ref{ceps}). With the aid of Proposition \ref{prohau}, there exist a positive constant $C(u_0)$ and finitely many balls $\{\mathbb B_{r_i}(x_i)\}$ with $x_i\in \mathcal{N}(u_0)$ and $r_i\leq \frac{1}{10}$, such that for any $u\in D^1_N$ and $\|u-u_0\|_{L^\infty (\mathbb B_{7/8})}\leq \epsilon_0$, there holds
\begin{align}
H^{n-1}\big( \mathcal{N}(u)\cap \mathbb B_{1/2}\backslash \bigcup\mathbb B_{r_i}(x_i)\big)\leq C(u_0)
\end{align}
and
$$ \sum r^{n-1}_i\leq \frac{1}{2}. $$
If $\mathcal{N}(u)\cap \mathbb B_{r_i}(x_i)\not =\emptyset $, we may take some point $\tilde{x}_i$ in $\mathcal{N}(u)\cap \mathbb B_{r_i}(x_i)$. Clearly, it holds that $\mathbb B_{r_i}(x_i)\subset \mathbb B_{2r_i}(\tilde{x}_i)$. We may rearrange the center and radius.
Thus, we can still find a finite collection of balls $\{\mathbb B_{r_i}(x_i)\}$ with $x_i\in \mathcal{N}(u)$ and $r_i\leq \frac{1}{4}$ such that
\begin{align}
H^{n-1}\big( \mathcal{N}(u)\cap \mathbb B_{1/2}\backslash \bigcup \mathbb B_{r_i}(x_i)\big)\leq C(u_0)
\end{align}
and
$$ \sum r^{n-1}_i\leq \frac{1}{2}. $$
Since $D_N^1$ is compact under the local $L^\infty$ norm, there exists $u_1, u_2, \cdots, u_p\in D^1_N$ and $\epsilon_1=\epsilon_1(u), \cdots,
\epsilon_p=\epsilon_p(u)$ such that for any $u\in D_N^1$, there exists a $1\leq k\leq p$ satisfying the property
$$ \|u-u_k\|_{L^\infty (\mathbb B_{7/8})}\leq \epsilon_k\leq \epsilon_0. $$
Denote
$$ C=\max\{ C(u_1), \cdots, C(u_p)\}.$$
This $C$ depends on the class of $D^\ast_N$. Thus, we complete the proof.
\end{proof}
Now we are ready to prove Theorem \ref{th4} in the section. We apply the standard iteration arguments in \cite{HS}.
\begin{proof}[Proof of Theorem \ref{th4}]
First, we define
$$\phi_0=\{\mathbb B_{1/2}(0)\}.$$
We claim that we can find $\phi_1, \phi_2, \cdots, $ each of which is a collection of balls such that
\begin{align}
rad(\mathbb B)\leq \frac{1}{2} (\frac{1}{2})^l \ \mbox{for any} \ \mathbb B\in \phi_l,
\label{induc1}
\end{align}
\begin{align}
\sum_{\mathbb B\in\phi_l} [rad(\mathbb B)]^{n-1}\leq (\frac{1}{2})^l,
\label{induc2}
\end{align}
and
\begin{align}
H^{n-1}\big( \mathcal{N}(u)\cap \bigcup_{\mathbb B\in \phi_{l-1}} \mathbb B \backslash \bigcup_{\mathbb B\in \phi_{l}} \mathbb B \big)\leq C(\frac{1}{2})^{l-1}
\label{induc3}
\end{align}
for $l\geq 1$, where $C$ is the positive constant in Lemma \ref{le13}. We prove the claim by constructing $\{\phi_l\}$ using induction.
Note that $\phi_0=\{\mathbb B_{1/2}(0)\}$. Suppose that the assumptions (\ref{induc1})--(\ref{induc3}) hold for $l-1$. We construct $\phi_l$. Taking any $\mathbb B= \mathbb B_r(y)\in \phi_{l-1}$, by the transformation $x\to y+2rx$, via $Lu=0$ in $\mathbb B_{2r}(y)$, we have
$\hat{L} \hat{u}=0$ in $\mathbb B_1$ with
$$ \hat{L}=\sum^4_{|\nu|=0} (2r)^{4-|\nu|} a_{\nu}(y+2rx) D_x^\nu $$
and $\hat{u}(x)=u(y+2rx)$. We observe that $\hat{u}\in D^\ast_N$. Applying Lemma \ref{le13}, we obtain a collection of balls $\{\mathbb B_{s_i}(z_i)\}$ with $s_i\leq \frac{1}{4}$ and $z_i\in \mathcal{N}(\hat{u})$ such that
$$ H^{n-1}\big( \mathcal{N}(\hat{u})\cap \mathbb B_{1/2}\backslash \mathbb B_{s_i}(z_i)\big)\leq C   $$
and
$$ \sum s_i^{n-1}\leq \frac{1}{2}.  $$
Rescaling $\mathbb B_{1/2}(0)$ back to $\mathbb B_{r}(y)$ by $x \mapsto \frac{x-y}{2r}$ gives that, for $\mathbb B=\mathbb B_r(y)\in \phi_{l-1}$, there exist finitely many balls $\{\mathbb B_{r_i}(x_i)\}$ in $\mathbb B_{2r}(y)$ with $r_i\leq \frac{r}{2}$, such that
\begin{align*}
H^{n-1}\big( \mathcal{N}({u})\cap \mathbb B_{r}(y)\backslash \bigcup\mathbb B_{r_i}(x_i)\big)\leq C  r^{n-1}
\end{align*}
and
$$ \sum r_i^{n-1}\leq \frac{1}{2} r^{n-1}.  $$
For such $\mathbb B_r(y)$, we set
$$ \phi^B_{l}=\bigcup \{\mathbb B_{r_i}(x_i)\}  $$
and construct $\phi_l$ as
$$ \phi_l=\bigcup_{\mathbb B\in \phi_{l-1}} \phi_l^B. $$
Applying Lemma \ref{le13} gives that
\begin{align}
H^{n-1}\big( \mathcal{N}(u)\cap \bigcup_{\mathbb B\in \phi_{l-1}} \mathbb B \backslash \bigcup_{\mathbb B\in \phi_{l}} \mathbb B \big)\leq C\big( \sum_{\mathbb B_{r_i}(x_i)\in\phi_{l-1}} r_i^{n-1}\big).
\end{align}
By induction, we obtain that, for $\mathbb B_{r_i}\in \phi_l$,
\begin{align}
r_i\leq \frac{1}{2} (\frac{1}{2})^l, \quad \  \sum_{\mathbb B_{r_i}(x_i)\in\phi_{l}} r_i^{n-1}\leq (\frac{1}{2})^l,
\label{show1}
\end{align}
and
\begin{align}
H^{n-1}\big( \mathcal{N}(u)\cap \bigcup_{\mathbb B\in \phi_{l-1}} \mathbb B \backslash \bigcup_{\mathbb B\in \phi_{l}} \mathbb B \big)\leq C(\frac{1}{2})^{l-1}.
\label{show2}
\end{align}
Thus, we have shown the claim (\ref{induc1})--(\ref{induc3}).

Since
\begin{align*}
\mathcal{N}(u)\cap\mathbb B_{1/2}(0)\subset \bigcup^\infty_{l=1}\big( \mathcal{N}(u)\cap \bigcup_{\mathbb B\in \phi_{l-1}} \mathbb B \backslash \bigcup_{\mathbb B\in \phi_{l}} \mathbb B   \big) \cup \bigcap_{l=0}\big(\mathcal{N}(u)\cap \bigcup^\infty_{j=l} \bigcup_{\mathbb B\in \phi_j} \mathbb B   \big),
\end{align*}
it follows from (\ref{show1}) and (\ref{show2}) that
\begin{align}
H^{n-1}\big(\mathcal{N}(u)\cap\mathbb B_{1/2}(0)\big) \leq C \{ \sum_{l\geq 1}(\frac{1}{2})^{l-1}+\inf_{l\geq 1} \sum^\infty_{j=l} (\frac{1}{2})^j\}\leq C.
\end{align}
Therefore, we conclude that
\begin{align}
H^{n-1}(x\in \mathbb B_{1/2}| u=0)\leq C(N).
\label{immi}
\end{align}
From Theorem \ref{th3}, we learn that the doubling inequality (\ref{DNDN}) holds for any $x_0\in \mathcal{M}$ with $N\leq CM^\frac{1}{3}$. Thus, it follows from (\ref{immi}) that
 \begin{align}
H^{n-1}(x\in \mathbb B_{r_0}| u=0)\leq C(r_0, M).
\end{align}
for any $\mathbb B_{r_0} \in  \mathcal{M}$. Since the manifold $\mathcal{M}$ is compact, by covering the manifold by finitely many balls, we can derive the conclusion in Theorem \ref{th4}.
\end{proof}

\section{Quantitative Cauchy uniqueness}

We prove a propagation of smallness results for bi-Laplace equations (\ref{bi-Laplace}) in this section. The similar results for second order elliptic equations have been shown by Lin in \cite{Lin}, where the proof is a little sketchy. We provide the detailed proof with a somewhat different argument using the Carleman estimates inspired by \cite{LR} and \cite{JL}. Similar results in terms of the $L^\infty$ norm can be obtained by using three spheres inequality repetitively from frequency function, see \cite{ARRV}.
Such results  play an important role not only in characterizing the doubling index in a cube in \cite{Lo}, but also in inverse problems. Using the Carleman estimates, we are able to show a two half-ball and one lower dimensional ball type result.
\begin{lemma}
Let $(u, v)$ be a solution of (\ref{system1}) in the half-ball $\mathbb B^+_1$. Denote
$$\tilde{\Gamma}=\{ (x', \ 0)\in \mathbb R^n| x'\in \mathbb R^{n-1}, \ |x'|<\frac{1}{3}\}.$$
Assume that
\begin{equation}
\|(u, v))\|_{H^1(\tilde {\Gamma})}+\|\partial_n{(u, v)}\|_{L^2(\tilde{\Gamma})}\leq \epsilon <<1
\end{equation}
and $\|(u, v)\|_{L^2(\mathbb B^+_\frac{1}{2})}\leq 1$. There exist positive constants $C$ and $\beta$ such that
\begin{equation}
\|(u, v)\|_{L^2(\frac{1}{256}\mathbb B^+_1)}\leq C \epsilon^\beta.
\label{too}
\end{equation}
More precisely, we can show that there exists $0<\gamma<1$ such that
\begin{equation}
\|(u, v)\|_{L^2(\frac{1}{256}\mathbb B^+_1)}\leq \|(u, v)\|_{L^2(\mathbb B^+_\frac{1}{2})}^\gamma \big( \|(u, v)\|_{H^1(\tilde {\Gamma})}+\|\partial_n{(u, v)}\|_{L^2(\tilde{\Gamma})}  \big)^{1-\gamma}.
\label{tool}
\end{equation}
\label{halfspace}
\end{lemma}
\begin{proof}
Our tools are some Carleman estimates in the half ball $\mathbb B^+_1=\{x\in \mathbb R^n| x\in \mathbb B_1 \ \mbox{and} \ x_n\geq 0\}$. For simplicity, we first establish such Carleman estimates for scalar functions.
We select a weight function $$\phi(x)=-\frac{|x'|^2}{4}+\frac{x_n^2}{2}-x_n,$$ where $x'=\{ x_1, x_2, \cdots, x_{n-1}\}.$
We consider $\phi(x)$ for $|x|$ in $\mathbb B^+_{\frac{1}{4}}$.

 Define $$\triangle_\tau g = e^{\tau\phi} \triangle (e^{-\tau\phi} g)$$
for $g\in C^\infty_0(\mathbb B^+_1)$.
Direct computations show that
$$ \triangle_\tau g=\triangle g-2\tau \nabla \phi\cdot \nabla g-\tau\triangle \phi g+\tau^2 |\nabla \phi|^2 g.  $$
We split $\triangle_\tau g$ into symmetric parts and anti-symmetric parts:
$$S_{\phi} g=\triangle g+\tau^2|\nabla \phi|^2 g,   $$
$$A_{\phi} g=-2\tau\nabla \phi \cdot \nabla g -\tau \triangle \phi g.   $$
Then
\begin{equation}
\|\triangle_\tau g\|^2= \|S_{\phi} g\|^2+\| A_{\phi} g\|^2+ 2\langle S_{\phi} g, \ A_{\phi} g\rangle.
\label{kuma}
\end{equation}
We study the inner product term $\langle S_{\phi} g, \ A_{\phi} g\rangle$.
Note that
$$\nabla \phi= \langle -\frac{x'}{2}, \ x_n-1\rangle, \quad \triangle \phi=\frac{-n+3}{2}.   $$
We can check that
\begin{align}
\langle S_{\phi} g, \ A_{\phi} g\rangle&=\langle \triangle g+\tau^2|\nabla \phi|^2 g, \ -2\tau\nabla \phi \cdot \nabla g -\tau \triangle \phi g\rangle \nonumber\\
&=\langle \triangle g+\tau^2 \frac{|x'|^2}{4}g+ \tau^2(1-x_n)^2 g, \quad 2\tau(1-x_n)\frac{\partial g}{\partial x_n}+\tau x' \cdot \nabla' g+\frac{n-3}{2}\tau g \rangle,
\label{nogain}
\end{align}
where $\nabla'g= \langle \frac{\partial g}{\partial x_1}, \cdots, \frac{\partial g}{\partial x_{n-1}}\rangle$. We estimate each term in the inner product using integration by parts argument. Integrating by parts twice shows that
\begin{align}
\langle \triangle g, \ 2\tau(1-x_n)\frac{\partial g}{\partial x_n}\rangle&=2\tau \int_{\mathbb B^+_r} (\frac{\partial g}{\partial x_n})^2\,dx-
\tau\int_{\mathbb B^+_r}(1-x_n)\frac{\partial}{\partial x_n}|\nabla g|^2\,dx+ 2\tau \int_{\bf B_r} (\frac{\partial g}{\partial x_n})^2 \,dx' \nonumber\\
&=2\tau \int_{\mathbb B^+_r} (\frac{\partial g}{\partial x_n})^2\,dx-\tau\int_{\mathbb B^+_r}|\nabla g|^2\,dx -\tau \int_{\bf B_r} |\nabla g|^2\,dx'\nonumber\\ &+2\tau \int_{\bf B_r} (\frac{\partial g}{\partial x_n})^2\,dx',
\label{tan1}
\end{align}
where ${\bf B_r}$ is the ball centered at origin with radius $r$ in $\mathbb R^{n-1}$. It follows from integration by parts that
\begin{align}
\langle \triangle g, \ \tau x'\cdot\nabla' g \rangle&=-\tau\int_{\mathbb B^+_r} |\nabla' g|^2\,dx-\tau\int_{\mathbb B^+_r} \sum^{n-1}_{i,j=1} x_j \frac{\partial^2 g}{\partial x_i \partial x_j} \frac{\partial g}{\partial x_i}\, dx \nonumber \\ &-\tau \int_{\mathbb B^+_r}\sum^{n-1}_{i=1} x_i \frac{\partial^2 g}{\partial x_i \partial x_n} \frac{\partial g}{\partial x_n}\, dx+\tau \int_{\bf B_r} \frac{\partial g}{\partial x_n} x'\cdot \nabla' g\, dx'.\label{mis}
\end{align}
We consider the second term on the right hand side of last identity using integration by parts,
\begin{align}
-\tau\int_{\mathbb B^+_r} \sum^{n-1}_{i,j=1} x_j \frac{\partial^2 g}{\partial x_i \partial x_j} \frac{\partial g}{\partial x_i}\, dx&=-\frac{\tau}{2}
\int_{\mathbb B^+_r} x'\cdot \nabla'|\nabla' g|^2\, dx \nonumber \\
&=\frac{(n-1)\tau}{2}
\int_{\mathbb B^+_r} |\nabla' g|^2\, dx.
\end{align}
Applying the similar strategy to the following integral gives that
\begin{align}
-\tau \int_{\mathbb B^+_r}\sum^{n-1}_{i=1} x_i \frac{\partial^2 g}{\partial x_i \partial x_n} \frac{\partial g}{\partial x_n}\, dx=\frac{(n-1)\tau}{2} \int_{\mathbb B^+_r}|\frac{\partial g}{\partial x_n}|^2\, dx.
\label{mis1}
\end{align}
Combining (\ref{mis})--(\ref{mis1}) leads to
\begin{align}
\langle \triangle g, \ \tau x'\cdot\nabla' g \rangle=-\tau \int_{\mathbb B^+_r} |\nabla' g|^2\, dx+ \frac{(n-1)\tau}{2}
\int_{\mathbb B^+_r} |\nabla g|^2 \, dx+\tau  \int_{\bf B_r} \frac{\partial g}{\partial x_n} x'\cdot \nabla' g\, dx'.
\label{Omy}
\end{align}

Taking (\ref{tan1}) and (\ref{Omy}) into consideration yields that
\begin{align}
\langle \triangle g,  -2\tau\nabla \phi\cdot \nabla g\rangle &= \frac{(n-3)\tau}{2}
\int_{\mathbb B^+_r} |\nabla g|^2 \, dx-\tau \int_{\mathbb B^+_r} |\nabla' g|^2 \, dx+2\tau \int_{\mathbb B^+_r} |\frac{\partial g}{\partial x_n} |^2 \, dx \nonumber \\
&+\tau  \int_{\bf B_r} \frac{\partial g}{\partial x_n} x'\cdot \nabla' g\, dx'-\tau \int_{\bf B_r} |\nabla' g|^2 \, dx'+\tau \int_{\bf B_r} |\frac{\partial g}{\partial x_n}|^2 \, dx'.
\label{ren}
\end{align}

We proceed to consider other terms in (\ref{nogain}).
Integration by parts argument shows that
\begin{align}
\langle \tau^2 \frac{|x'|^2}{4} g, \ 2\tau(1-x_n)\frac{\partial g}{\partial x_n}\rangle =\frac{\tau^3}{4} \int_{\mathbb B^+_r} |x'|^2 g^2\, dx+\frac{\tau^3}{4}\int_{\ \textbf{B}_r}|x'|^2 g^2\, dx' .
\label{fool1}
\end{align}
Furthermore, we get
\begin{align}
\langle \tau^2 \frac{|x'|^2}{4} g, \ \tau x'\cdot \nabla' g \rangle=-\frac{n+1}{8}\tau^3 \int_{\mathbb B^+_r}|x'|^2 g^2 \, dx.
\label{fool2}
\end{align}
\begin{align}
\langle \tau^2 (x_n-1)^2 g, \  \tau x'\cdot \nabla' g\rangle=\frac{-(n-1)}{2}\tau^3 \int_{\mathbb B^+_r} (x_n-1)^2 g^2\, dx.
\label{fool3}
\end{align}

\begin{align}
\langle \tau^2 (x_n-1)^2 g, \ 2\tau(1-x_n) \frac{\partial g}{\partial x_n} \rangle =3\tau^3 \int_{\mathbb B^+_r} (x_n-1)^2 g^2\, dx +\tau^3 \int_{\ \textbf{B}_r}  g^2\, dx'.
\label{fool4}
\end{align}

Together with the estimates (\ref{fool1})--(\ref{fool4}), we obtain that
\begin{align}
\langle \tau^2|\nabla \phi|^2 g, \ -2\tau\nabla \phi \cdot \nabla g \rangle&=\frac{1-n}{8}\tau^3 \int_{\mathbb B^+_r} |x'|^2 g^2\, dx+\frac{7-n}{2}\tau^3 \int_{\mathbb B^+_r} (x_n-1)^2 g^2\, dx\nonumber \\&+\tau^3 \int_{\ \textbf{B}_r}(1+\frac{|x'|^2}{4})  g^2\, dx'.
\label{ren2}
\end{align}

It is trivial to see that
\begin{align}
\langle \tau^2 (1-x_n)^2 g, \ \frac{n-3}{2}\tau g\rangle=\frac{n-3}{2}\tau^3 \int_{\mathbb B^+_r} (1-x_n)^2 g^2\,dx
\label{tan3}
\end{align}
and
\begin{align}
\langle \tau^2 \frac{|x'|^2}{4} g, \ \frac{n-3}{2}\tau g\rangle=\frac{n-3}{8}\tau^3 \int_{\mathbb B^+_r} {|x'|^2} g^2\, dx.
\label{ttt}
\end{align}
The combination of (\ref{tan3}) and (\ref{ttt}) yields that
\begin{align}
\langle \tau^2|\nabla \phi|^2 g,\ -\tau\triangle \phi g\rangle=\tau^3\int_{\mathbb B^+_r} (\frac{n-3}{8}{|x'|^2}+\frac{n-3}{2}(x_n-1)^2) g^2\, dx.
\label{ren3}
\end{align}

We are left to deal with the last inner product.
Performing the integration by parts argument shows that
\begin{align}
\langle \triangle g, \ -\tau\triangle\phi g\rangle&=
\langle \triangle g, \ \frac{n-3}{2}\tau g\rangle  \nonumber \\&=-\frac{n-3}{2}\tau \int_{\mathbb B^+_r}|\nabla g|^2 \,dx+\frac{n-3}{2}\tau \int_{\bf B_r} \frac{\partial g}{\partial x_n} g \,dx'.
\label{ren4}
\end{align}

Combining the identities (\ref{nogain}), (\ref{ren}), (\ref{ren2}), (\ref{ren3}) and (\ref{ren4}), we arrive at
\begin{align}
\langle S_{\phi} g, \ A_{\phi} g\rangle&= -\tau \int_{\mathbb B^+_r} |\nabla' g|^2 \, dx+2\tau \int_{\mathbb B^+_r} |\frac{\partial g}{\partial x_n} |^2 \, dx-\frac{1}{4}\tau^3 \int_{\mathbb B^+_r} |x'|^2 g^2\, dx \nonumber \\
&+2 \tau^3 \int_{\mathbb B^+_r} (x_n-1)^2 g^2\, dx+\tau^3 \int_{\bf B_r} (1+\frac{|x'|^2}{4})g^2\,dx'+\frac{n-3}{2}\tau  \int_{\bf B_r} \frac{\partial g}{\partial x_n} g \, dx' \nonumber \\
&+\tau  \int_{\bf B_r} \frac{\partial g}{\partial x_n} x'\cdot \nabla' g\, dx'-\tau \int_{\bf B_r} |\nabla' g|^2 \, dx'+\tau \int_{\bf B_r} |\frac{\partial g}{\partial x_n}|^2 \, dx'.
\end{align}
Since it is assumed that $r<\frac{1}{4}$, simple calculations indicate that
$$ \frac{1}{8}(x_n-1)^2-|x'|^2>0.  $$
By Cauchy-Schwartz inequality, we have

\begin{align}
\langle S_{\phi} g, \ A_{\phi} g\rangle&+C\tau \int_{\bf B_r} |\nabla' g|^2 \, dx'+C\tau \int_{\bf B_r} | g|^2 \, dx'+C\tau \int_{\bf B_r} |\frac{\partial g}{\partial x_n} |^2 \, dx'\nonumber \\
&\geq -\tau \int_{\mathbb B^+_r} |\nabla' g|^2 \, dx+2\tau \int_{\mathbb B^+_r} |\frac{\partial g}{\partial x_n} |^2 \, dx+\frac{63}{32} \tau^3 \int_{\mathbb B^+_r} (x_n-1)^2 g^2 \, dx+\tau^3 \int_{\bf B_r} g^2 \, dx'.
\label{anot}
\end{align}

We also want to include the gradient term in the Carleman estimates. To this end, we compute the following inner product with some small constant $\epsilon>0$ to be determined,
\begin{align}
\langle S_{\phi} g, \ -\frac{16(1+\epsilon)^2}{9}\tau(1-x_n)^2 g\rangle&= \langle \triangle g+\frac{|x'|^2}{4}\tau^2 g+ \tau^2(1-x_n)^2g, \ -\frac{16(1+\epsilon)^2}{9}\tau(1-x_n)^2 g\rangle
 \nonumber\\
&=\frac{16(1+\epsilon)^2}{9}\big(\tau \int_{\mathbb B^+_r}(1-x_n)^2|\nabla g|^2\,dx-2\tau \int_{\mathbb B^+_r} (1-x_n)\frac{\partial g}{\partial x_n} g\,dx\nonumber\\ & -\tau  \int_{\bf B_r} \frac{\partial g}{\partial x_n} g \, dx'
-\frac{\tau^3}{4}\int_{\mathbb B^+_r}|x'|^2(1-x_n)^2 g^2 \nonumber\\ &-\tau^3 \int_{\mathbb B^+_r}(1-x_n)^4 g^2 \, dx\big).
\end{align}
Thus, for $r<\frac{1}{4}$,
\begin{align}
\|S_{\phi} g\|^2+\|\frac{8(1+\epsilon)^2}{9}{\tau (1-x_n)^2 g}\|^2&\geq \langle S_{\phi} g, \ -\frac{16(1+\epsilon)^2}{9}\tau(1-x_n)^2 g\rangle\nonumber\\
&\geq (1+\epsilon)^2\tau \int_{\mathbb B^+_r}|\nabla g|^2\,dx-\frac{32(1+\epsilon)^2}{9}\tau \int_{\mathbb B^+_r} (1-x_n)\frac{\partial g}{\partial x_n} g\,dx \nonumber\\
&  -\frac{16(1+\epsilon)^2}{9}\tau  \int_{\bf B_r} \frac{\partial g}{\partial x_n} g \, dx'-\frac{\tau^3(1+\epsilon)^2}{9} \int_{\mathbb B^+_r}(1-x_n)^2 g^2\, dx \nonumber\\
& -\frac{16(1+\epsilon)^2 \tau^3}{9} \int_{\mathbb B^+_r}(1-x_n)^2 g^2 \, dx.
\label{mabi}
\end{align}
We choose  $\epsilon$ so small that
  \begin{align}
\frac{63}{32}- \frac{ (1+\epsilon)^2}{9}-\frac{16(1+\epsilon)^2}{9}>0.
  \end{align}
Combining the estimates (\ref{kuma}), (\ref{anot}), (\ref{mabi}),
  and using Cauchy-Schwartz inequality and  the fact that $r<\frac{1}{4}$, we get that
\begin{align}
&\|\triangle_\tau g\|^2 +C\tau \int_{\bf B_r} |\nabla' g|^2 \, dx'+C\tau \int_{\bf B_r} | g|^2 \, dx'+C\tau \int_{\bf B_r} |\frac{\partial g}{\partial x_n} |^2 \, dx' \nonumber\\
&+ \| \frac{8(1+\epsilon)^2}{9}\tau (1-x_n)^2 g\|^2  \nonumber\\ & \geq C \tau\int_{\mathbb B^+_r}(1-x_n)^2|\nabla g|^2 \, dx
+C\tau^3 \int_{\mathbb B^+_r} (1-x_n)^2 g^2\, dx+C\tau^3 \int_{\bf B_r} g^2 \, dx'
\label{barc}
\end{align}
for $\tau>\bar C$, where $\bar C$ depends only on $n$.
Since $\tau$ is a large constant, we can absorb the  fifth term on the left hand side of last inequality into the left hand side. Therefore, we get
\begin{align}
&\|\triangle_\tau g\|^2+ C\tau \int_{\bf B_r} |\nabla' g|^2 \, dx'+C\tau \int_{\bf B_r} (\frac{\partial g}{\partial x_n})^2\, dx'+ C\tau \int_{\bf B_r} g^2\, dx' \nonumber\\
&\geq C\tau\int_{\mathbb B^+_r}(1-x_n)^2|\nabla g|^2 \, dx +C\tau^3 \int_{\mathbb B^+_r} (1-x_n)^2 g^2\, dx. \
+C\tau^3 \int_{\bf B_r} g^2 \, dx'.
\label{endd}
\end{align}
Let $f= e^{-\tau \phi} g$.  The inequality (\ref{endd}) implies the desirable estimates
\begin{align}
&\| e^{\tau \phi} \triangle f\|_{L^2(\mathbb B_r^+)}+\tau^{\frac{1}{2}} \| e^{\tau \phi}  f  \|_{L^2(\bf B_r)}+
\tau^{\frac{1}{2}} \| e^{\tau \phi} \frac{\partial f}{\partial x_n} \|_{L^2(\bf B_r)}+ \tau^{\frac{1}{2}} \| e^{\tau \phi} \nabla' f \|_{L^2(\bf B_r)} \nonumber\\
&\geq C\tau^{\frac{3}{2}} \| e^{\tau \phi}(1-x_n) f  \|_{L^2(\mathbb B_r^+)}+C\tau^{\frac{1}{2}}\| e^{\tau \phi}(1-x_n)\nabla f  \|_{L^2(\mathbb B_r^+)}.
\end{align}
By the similar argument, it also holds for a vector function $F=(f_1, \ f_2)$. That is,
\begin{align}
&\| e^{\tau \phi} \triangle F\|_{L^2(\mathbb B_r^+)}+\tau^{\frac{1}{2}} \| e^{\tau \phi}  F \|_{L^2(\bf B_r)}+
\tau^{\frac{1}{2}} \| e^{\tau \phi} \frac{\partial F}{\partial x_n} \|_{L^2(\bf B_r)}+ \tau^{\frac{1}{2}} \| e^{\tau \phi} \nabla' F \|_{L^2(\bf B_r)} \nonumber\\
&\geq C\tau^{\frac{3}{2}} \| e^{\tau \phi}(1-x_n) F \|_{L^2(\mathbb B_r^+)}+C\tau^{\frac{1}{2}}\| e^{\tau \phi}(1-x_n)\nabla F  \|_{L^2(\mathbb B_r^+)}.
\label{usee}
\end{align}
 The following Caccioppolli inequality holds for the solutions of (\ref{system1}) in $\mathbb B^+_1$,
\begin{align}
 \|\nabla U\|_{L^2(\mathbb B_r^+)} \leq \frac{C}{r}\big ( \| U\|_{L^2(\mathbb B_{2r}^+)}+ \|\frac{\partial U}{\partial x_n}\|_{L^2({\bf B}_{2r})}+ \|{ U}\|_{L^2({\bf B}_{2r})}\big).
 \label{cacciop}
\end{align}

Let $\bar V(x)= \begin{pmatrix}
0, & 1 \\
\bar W(x), & 0
\end{pmatrix}
$. We select a smooth cut-off function $\eta$ such that $\eta(x)=1$ in $\mathbb B_{\frac{1}{8}}^+$ and $\eta(x)=0$ outside $\mathbb B_{\frac{1}{4}}^+$. Let $U=(u, v)^\intercal$. Substituting $F$ by $\eta U$ in the Carleman estimates (\ref{usee}) and then  the system (\ref{system1}) yields that
\begin{align}
&\| e^{\tau \phi}(\triangle \eta U+2\nabla \eta\cdot \nabla U )\|_{L^2(\mathbb B_\frac{1}{2}^+)}+\tau^{\frac{1}{2}} \| e^{\tau \phi}  \eta U \|_{L^2({\bf B}_\frac{1}{2})}+
\tau^{\frac{1}{2}} \| e^{\tau \phi} \frac{\partial (\eta U) }{\partial x_n} \|_{L^2({\bf B}_\frac{1}{2})} \nonumber\\ \quad &+ \tau^{\frac{1}{2}} \| e^{\tau \phi} \nabla' (\eta U)\|_{L^2({\bf B}_\frac{1}{2})}\nonumber\\
&\geq C\tau^{\frac{3}{2}} \| e^{\tau \phi}(1-x_n) \eta U \|_{L^2(\mathbb B_\frac{1}{2}^+)}.
\label{meimei}
\end{align}

We want to find the maximum of $\phi$ in the first term on the left hand side of (\ref{meimei}). Since $\phi$ is negative and decreasing with respect to $x'$ and $x_n$ for $r<\frac{1}{4}$, then
$$ \max_{\{\frac{1}{8}\leq r\leq \frac{1}{4} \}\cap \{x_n\geq 0\}}\phi =\max_{\{\frac{1}{8}\leq r\leq \frac{1}{4} \}} -\frac{|x'|^2}{4}=-\frac{1}{256}.$$
We also need to find a lower bound of $\phi$ for the term on the right hand side of (\ref{meimei}) such that $-\phi(x)<\frac{1}{256}$. Let
$$\hat{\phi}(a)=-\frac{a^2}{4}+\frac{a^2}{2}-a=\frac{a^2}{4}-a.  $$
Since $\phi$ decreases with respect to $x'$ and $x_n$, then the minimum of $\phi(x)$ is $\hat{\phi}(a)$ for $r<a$. Solving the inequality $-\hat{\phi}(a)<\frac{1}{256}$, we have one solution $a=\frac{1}{256}$.
Set $$\phi_0=\frac{1}{256}+\hat{\phi}(\frac{1}{256})>0,$$
$$\phi_1=\hat{\phi}(\frac{1}{256})<0.$$

Applying the Caccioppolli inequality (\ref{cacciop}), we arrive at
\begin{align}
e^{-\frac{\tau}{256}  } \|U\|_{L^2(\mathbb B_\frac{1}{2}^+)}+ \|U\|_{L^2({\bf B}_\frac{1}{3})}&+\|\nabla' U\|_{L^2({\bf B}_\frac{1}{3})}+\|\frac{\partial U}{\partial x_n}\|_{L^2({\bf B}_\frac{1}{3})} \nonumber\\
&\geq C\tau \| e^{\tau \phi}(1-x_n)\eta U  \|_{L^2(\mathbb B_\frac{1}{4}^+)}\nonumber\\
&\geq  C\tau e^{\tau  \hat{\phi}(\frac{1}{256})}\|U\|_{L^2(\mathbb B_{\frac{1}{256}}^+)}.
\end{align}
Let
$$   B_1=\|U\|_{L^2(\mathbb B_\frac{1}{2}^+)}, $$
$$  B_2=  \|U\|_{L^2({\bf B}_\frac{1}{3})}+\|\nabla' U\|_{L^2({\bf B}_\frac{1}{3})}+\|\frac{\partial U}{\partial x_n}\|_{L^2({\bf B}_\frac{1}{3})}, $$
$$ B_3=\|U\|_{L^2(\mathbb B_{\frac{1}{256}}^+)}. $$
Multiplying both sides of the last inequality by $e^{-\tau  \hat{\phi}(\frac{1}{256})}$ leads to
\begin{equation}
e^{-\tau  \phi_0  }B_1+ e^{-\tau  \phi_1  }B_2\geq C B_3.
\label{BBB}
\end{equation}
We introduce a parameter
$$\tau_0=\frac{\ln \frac{B_2}{B_1}}{\phi_1-\phi_0}.$$
If $\tau_0>\bar C$, where $\bar C$ is given for the validity of the estimates (\ref{barc}),  then we choose $\tau=\tau_0$ in (\ref{BBB}). Thus,
\begin{equation}
B_1^{\frac{\phi_1}{\phi_1-\phi_0}} B_2^{\frac{-\phi_0}{\phi_1-\phi_0}}\geq CB_3.
\end{equation}
Let $\gamma=\frac{\phi_1}{\phi_1-\phi_0}$. Then the following three balls type inequality follows
\begin{equation}
\|(u, v)\|_{L^2(\frac{1}{256}\mathbb B^+_1)}\leq \|(u, v)\|_{L^2(\mathbb B^+_\frac{1}{2})}^\gamma \big( \|(u, v)\|_{H^1(\tilde {\Gamma})}+\|\partial_n{(u, v)}\|_{L^2(\tilde{\Gamma})}  \big)^{1-\gamma}.
\label{nice}
\end{equation}
If $\tau_0\leq \bar C$, since $\phi_1-\phi_0$ is negative, then $B_2\geq C B_1$. It is clear that $B_3\leq B_1$. Again, we arrive at
$$B_3\leq C B_1^\gamma B_2^{1-\gamma}.  $$
Therefore, we show the estimates (\ref{nice}) again.  The estimate (\ref{too}) is a consequence of (\ref{nice}). The lemma is finished.
\end{proof}

\section{Appendix}
The following lemma serves as the starting step for the iteration argument in the proof of Proposition \ref{pro2}. For the scalar equations, the finite bound of $F(N)$ is established in \cite{HS}. See also a different proof using compactness arguments in \cite{HL1}. Instead of pursuing the nodal sets comparison lemma in \cite{HS}, we adapt the proof of Theorem \ref{th4} and the measure of rank zero sets of harmonic maps in  \cite{HL2} where elliptic systems are considered. We give the main ideas in the proof of the lemma.
\begin{lemma}
Let $(u,v)$ be the solution in (\ref{system1}) and $F(N)$ be defined in (\ref{defnew}). Then $F(N)<C(N)$.
\label{bounf}
\end{lemma}
\begin{proof}
Following the arguments in \cite{Han}, the set $\{\mathbb B_{1/2}|u=v=0\}$ is countably $(n-1)$-rectifiable.  Let $y$ be in the $n-1$ dimensional nodal set $\{\mathbb B_{1/2}|u=v=0\}$. There exist leading monomials $P_y^1$ and $P_y^2$ such that
\begin{align}
\triangle P_y^1=0,  \quad  \triangle P_y^2=0.
\end{align}
By an appropriate rotation, we can have either
\begin{align}
u(x)=C_1 x_1+ \psi_1(x)
\end{align}
or
\begin{align}
v(x)=C_2 x_1+ \psi_2(x),
\end{align}
where
\begin{align*}
|\psi_i(x)|\leq C|x|^{1+\alpha_i} \quad \mbox{for some} \  0< \alpha_i<1, \quad i=1, 2.
\end{align*}

Following the proof of Proposition \ref{prohau}, we can show that, there exist positive constants $C(u, v)$ and $\epsilon(u, v)$ and a finite collection of balls $\{\mathbb B_{r_i}(x_i)\}$ with $r_i\leq \frac{1}{8}$ and $x_i\in \{\mathbb B_{1/2}|u=v=0\}$ such that, for $(u_1, v_1)\in C^1$ with
\begin{align}
\|(u, v)- (u_1, v_1)\|_{C^1(\mathbb B_1)}\leq \epsilon(u, v),
\label{somefur}
\end{align}
there hold
\begin{align}
H^{n-1}\big(\{u_1=v_1=0\}\cap \mathbb B_{1/2}\backslash \bigcup B_{r_i}(x_i)\big)< C(u, v)
\label{somemore}
\end{align}
and
$$\sum r_i^{n-1} \leq \frac{1}{2^n}.  $$

The key to prove (\ref{somemore}) is to show that
\begin{align}
H^{n-1}\big(\{u_1=v_1=0\}\cap \mathbb B_{r}(y)\big)< C(u, v)r^{n-1} \label{somekey}
\end{align}
under the condition (\ref{somefur}). It follow from (\ref{somefur}) and the arguments in the proof of proposition \ref{th3} that there holds
\begin{align}
H^{n-1}\big(u_1^{-1}(0)\cap \mathbb B_r\big)< C_1(u, v) r^{n-1} \quad \mbox{or} \quad H^{n-1}\big(v_1^{-1}(0)\cap \mathbb B_r\big)< C_2(u, v) r^{n-1}.
\end{align}
Since
\begin{align}
\{u_1=v_1=0\}\cap \mathbb B_{r}(y)\subset u^{-1}_1(0)\cap \mathbb B_{r}(y) \quad \mbox{or} \quad \{u_1=v_1=0\}\cap \mathbb B_{r}(y)\subset v^{-1}_1(0)\cap \mathbb B_{r}(y),
\end{align}
the estimate (\ref{somekey}) follows.
Because of $N_{(u, v)}(Q)\leq N$,  the doubling inequality holds
\begin{equation}
\int_{\mathbb B_{2r}(x_0)} u^2+v^2\,dx\leq e^{CN}  \int_{\mathbb B_{r}(x_0)} u^2+v^2\,dx.
\end{equation}
Following the arguments of the proof of Theorem \ref{th4}, we can show that, there exists $C(N)$ depending on $N$ such that
\begin{align}
H^{n-1}(\mathbb B_{1/2}|u=v=0)\leq C(N).
\end{align}
This completes the proof of the lemma.
\end{proof}

\end{document}